\documentclass[11pt,a4paper,reqno]{amsart}
\usepackage{graphicx,amsmath,amsfonts,amssymb,amsthm,tcolorbox,dsfont} 

\usepackage{xcolor}
\usepackage[colorlinks=true, linkcolor=blue, citecolor=red]{hyperref}
\hypersetup{
    colorlinks,
    linkcolor={blue!60!black},
    citecolor={red!60!black},
    urlcolor={blue!80!black}
}
\usepackage{cleveref}
\usepackage{tabularx}
\usepackage{float}
\usepackage{nicefrac}

\calclayout

\usepackage[
backend=biber,
style=alphabetic,
sorting=ynt,
maxbibnames=10,
maxalphanames=4
]{biblatex}

\usepackage[]{appendix}
\numberwithin{equation}{section}

\newtheorem{theorem}{Theorem}[section]
\newtheorem*{theorem*}{Theorem}
\newtheorem{lemma}{Lemma}[section]
\newtheorem{prop}[lemma]{Proposition}

\theoremstyle{definition}
\newtheorem{definition}[lemma]{Definition}

\theoremstyle{remark}
\newtheorem{remark}[lemma]{Remark}

\newcommand{\vertiii}[1]{{\left\vert\kern-0.25ex\left\vert\kern-0.25ex\left\vert #1 
    \right\vert\kern-0.25ex\right\vert\kern-0.25ex\right\vert}}

\newcommand{\bu}{\mathbf{u}}
\newcommand{\bv}{\mathbf{v}}
\newcommand{\bw}{\mathbf{w}}
\newcommand{\bF}{\mathbf{F}}
\newcommand{\br}{\mathbf{r}}
\renewcommand{\d}{\, \mathrm{d}}
\newcommand{\der}{\mathrm{d}}
\newcommand{\esssup}{\mathrm{ess \, sup}}
\newcommand{\essinf}{\mathrm{ess \, inf}}

\newcommand{\loc}{\mathrm{loc}}
\newcommand{\bM}{\mathbf{M}}
\newcommand{\bzeta}{\boldsymbol{\zeta}}
\renewcommand{\leq}{\leqslant}
\renewcommand{\geq}{\geqslant}

\newcommand*{\defeq}{\mathrel{\vcenter{\baselineskip0.5ex \lineskiplimit0pt
                     \hbox{\scriptsize.}\hbox{\scriptsize.}}}%
                     =}

\newcommand*{\eqdef}{=\mathrel{\vcenter{\baselineskip0.5ex \lineskiplimit0pt
                     \hbox{\scriptsize.}\hbox{\scriptsize.}}}%
                     }

\DeclareMathOperator\artanh{artanh}
\DeclareMathOperator\arcoth{arcoth}

\DeclareMathOperator\supp{supp}
\DeclareMathOperator\sgn{sgn}

\usepackage{xcolor}
\definecolor{englishgreen}{rgb}{0.0, 0.5, 0.0}

\title[Carrollian Fluids in One Dimension]{One-Dimensional Carrollian Fluids III: Global Existence and Weak Continuity in $L^\infty$}

\author[P.~M.~Petropoulos]{Marios Petropoulos}
\address[P.~M.~Petropoulos]{Centre de Physique Th\'eorique, Ecole Polytechnique, 91120 Palaiseau, France}
\email{marios.petropoulos@polytechnique.edu}

\author[S.~M.~Schulz]{Simon Schulz}
\address[S.~M.~Schulz]{Scuola Normale Superiore, CRM De Giorgi, P.zza Cavalieri, 3,  56126 Pisa, Italy}\email{simon.schulz@sns.it}

\author[G.~Taujanskas]{Grigalius Taujanskas}
\address[G.~Taujanskas]{Faculty of Mathematics, Wilberforce Road,
Cambridge CB3 0WA,
UK}
\email{taujanskas@dpmms.cam.ac.uk}



\thanks{\emph{Centre de Physique Th\'eorique Preprint Number.} CPHT-RR028.052024.}

\begin{document}

\begin{abstract}

The Carrollian fluid equations arise as the $c \to 0$ limit of the relativistic fluid equations and have recently experienced a surge of activity in the flat-space holography community. However, the rigorous mathematical well-posedness theory for these equations does not appear to have been previously studied. This paper is the third in a series in which we initiate the systematic analysis of the Carrollian fluid equations. In the present work we prove the global-in-time existence of bounded entropy solutions to the isentropic Carrollian fluid equations in one spatial dimension for a particular constitutive law ($\gamma = 3$). Our method is to use a vanishing viscosity approximation for which we establish a compensated compactness framework. Using this framework we also prove the compactness of entropy solutions in $L^\infty$, and establish a kinetic formulation of the problem. This global existence result in $L^\infty$ extends the $C^1$ theory presented in \cite{AthanasiouPetropoulosSchulzTaujanskas24a}.

\end{abstract}

\maketitle
\thispagestyle{empty}


\setcounter{tocdepth}{1}
\tableofcontents

\section{Introduction}

This paper is the third in the series \cite{AthanasiouPetropoulosSchulzTaujanskas24a,AthanasiouPetropoulosSchulzTaujanskas24b} in which we initiate the rigorous mathematical analysis of the flat Carrollian fluid equations. This work is concerned with the global-in-time existence theory for the $2 \times 2$ system 
\begin{equation}\label{eq:carroll_eqns_i}
    \left\lbrace\begin{aligned}
        &\partial_t (\beta \sigma) + \partial_x \sigma = 0, \\ 
        & \partial_t \left(\gamma^{-1} \sigma^\gamma + \beta^2 \sigma\right) + \partial_x (\beta \sigma) = 0,
    \end{aligned}\right.
\end{equation}
which we call the \emph{isentropic Carrollian fluid equations}. Here $\gamma$ is a given positive constant, and the meaning and origin of the unknown quantities $\sigma$ and $\beta$ are explained below, as are the concepts of flat and isentropic. Our main result is the following (see \S\ref{sec:main_results} for detailed statements).

\begin{theorem*}
    Let $\gamma = 3$. For admissible initial data $(\sigma_0, \beta_0) \in L^\infty(\mathbb{R})$ there exists a global entropy solution $(\sigma, \beta) \in L^\infty(\mathbb{R}^2_+)$ to the system \eqref{eq:carroll_eqns_i}. Moreover, entropy solutions to \eqref{eq:carroll_eqns_i} are weakly compact in $L^\infty(\mathbb{R}^2_+)$, and the system \eqref{eq:carroll_eqns_i} admits a kinetic formulation.
\end{theorem*}

The Carrollian fluid equations are the speed of light $c \to 0$ limit of the relativistic fluid equations $\nabla_a T^{ab} = 0$, where $T_{ab}$ is the standard energy-momentum tensor for relativistic fluids \cite{CiambelliMarteauPetkouPetropoulosSiampos2018,PetkouPetropoulosRiveraBetancourSiampos2022,AthanasiouPetropoulosSchulzTaujanskas24a,Hartong2015,deBoerHartongObersSybesmaVandoren2018}. The $c \to 0$ \.In\"on\"u--Wigner contraction of the Poincar\'e group was first discovered by L\'evy-Leblond \cite{LevyLeblond1965} in 1965, who suggested the name \emph{Carroll group}, and shortly after independently by Sen Gupta \cite{SenGupta1966}. The motivations of Sen Gupta and L\'evy-Leblond appear to have been mostly pedagogical, but a decade and a half later their ``degenerate cousin of the Galilei group" turned out  to have some relevance in physics  
\cite{Hennaux1979}. It has recently resurfaced in the context of attempting to understand flat-space holography along with asymptotic symmetries in general relativity. In $d+1$ dimensions the Carrollian limit gives rise to a degenerate $(d+1)$-dimensional metric, of the kind which naturally occurs on the null boundary of an asymptotically flat $(d+2)$-dimensional spacetime (see \textit{e.g.} \cite{CampoleoniDelfantePekarPetropoulosRiveraBetancourVilatte2023} and references therein). As a consequence, studies of the associated \emph{Carrollian geometry} have recently seen a dramatic surge of activity  \cite{DuvalGibbonsHorvathyZhang2014,DuvalGibbonsHorvathy2014,DuvalGibbonsHorvathy2014b,BekaertMorand2016,BekaertMorand2018,Morand2020,CiambelliLeighMarteauPetropoulos2019,Herfray2022}. At the level of the equations of motion, the Carrollian limit of relativistic fluids formally yields a set of $2d + 1$ conservation laws for the total energy density, an energy flux, a symmetric stress-energy tensor, a momentum density, and a symmetric super-stress-energy tensor. Through the correspondence between relativistic fluids on the boundary of an anti-de Sitter spacetime and gravity in the bulk, these observables then formally correspond to the holographic duals of the gravitational field in the limit $\Lambda \to 0$, where $\Lambda$ is the cosmological constant. 

Notwithstanding their geometrical and physical significance, the rigorous mathematical underpinnings of the Carrollian fluid system are yet to be established, even for $d=1$. In this case the equations may be written as a $2 \times 2$ system of conservation laws: see \S2 of the companion paper \cite{AthanasiouPetropoulosSchulzTaujanskas24b}. Systems of hyperbolic conservation laws are an old and classical subject in the analysis of PDEs \cite{Dafermos2016,Serre1,Serre2}. Much of the theory of hyperbolic systems has focused on treating the dual limit of the relativistic fluid equations, \textit{i.e.}~$c \to \infty$, which yields the classical Galilean fluid equations: the compressible Euler equations. The existence theory for the Euler system in the case of constant thermodynamic entropy\footnote{One should not confuse the thermodynamic entropy, which will be constant in the systems at hand, and the PDE entropy introduced later as an index of monotonous evolution.} and polytropic fluids (as well as for more general constitutive relations) is now well-understood; see \textit{e.g.}~\cite{Chen86,DingChenLuo85,LionsPerthameTadmor94,LionsPerthameSouganidis96,ChenLeFloch,Matthew} and references therein. An effective classical method to construct the suitable notion of so-called \emph{entropy solutions} for the Galilean equations proceeds by the vanishing viscosity approximation, for which convergence can be shown using the compensated compactness framework of Tartar and Murat \cite{Tartar78,Tartar79}; the first application of this method to a system of conservation laws goes back to DiPerna \cite{DiPerna83a,DiPerna83b} in the 1980s. For Galilean fluids, on the one hand, proceeding by means of the vanishing viscosity method is physically reasonable, as one formally expects to recover the Euler equations from their viscous analogues, the compressible Navier--Stokes equations. For Carrollian fluids, on the other hand, physical intuition suggests that the $c \to 0$ limit should in some sense prohibit individual motion altogether, though without excluding collective time-dependent phenomena \cite{BergshoeffGomisLonghi2014}. In turn, it is not even a priori clear what the type of the Carrollian fluid system (in the sense of the usual classification of PDEs) is, or how a well-posedness theory should be constructed.

We focus in this paper on the case of Carrollian fluids on a flat background in dimension $d=1$ and study their global dynamics. In this case it is known that the $2d+1 = 3$ Carrollian equations are related to the three Galilean Euler equations by a duality map, and their symmetry groups are isomorphic \cite{AthanasiouPetropoulosSchulzTaujanskas24a,AthanasiouPetropoulosSchulzTaujanskas24b,DuvalGibbonsHorvathyZhang2014,Bagchi2010}. This duality map interchanges time and space and maps the Galilean fluid density $\rho$ to the \emph{Carrollian stress} $\sigma$, with dimensions of $\text{energy}\times\text{length}\times\text{time}^{-2}$, and the Galilean velocity $v$ to the \emph{Carrollian velocity} $\beta$, with dimensions of time over length, \textit{i.e.} inverse velocity. More generally, this duality interchanges equilibrium and out-of-equilibrium observables, and in our case a perfect Galilean fluid with a non-perfect Carrollian relative. The existence of this duality suggests that the Carrollian equations are indeed hyperbolic in some region in phase space, but the fact that solutions are understood on the Galilean side in fact does not provide information about the solutions on the Carrollian side. Indeed, the duality map is nonlinear, and a priori does not preserve regularity (\textit{cf.}~\cite{AthanasiouPetropoulosSchulzTaujanskas24b}). Furthermore, under the interchange of time and space the Galilean Cauchy problem is mapped to a Carrollian problem with \emph{boundary} data. Instead, here we wish to consider the Carrollian Cauchy problem.

As already mentioned, in the one-dimensional Carrollian fluids under consideration, a central dynamical variable besides the velocity $\beta$ is the Carrollian stress $\sigma$. The latter is in fact the super-stress-energy tensor \emph{i.e.} the order-$c^{-2}$ term of the parent relativistic stress. The term of order $c^{0}$ is the ordinary viscous stress, which combined with the pressure $p$ provides the so-called \emph{generalized Carrollian pressure} $\varpi$. In the same fashion, the relativistic heat current has an order-$c^{0}$ term, which we assume to vanish (this is a consequence of the Carroll--Galilei duality when the Galilean dynamics exhibits a genuine matter conservation, see \cite{AthanasiouPetropoulosSchulzTaujanskas24a} for details); and an order-$c^{2}$ term $\pi$, called here \emph{Carrollian heat current}.
The product of the Carrollian velocity $\beta$ and the Carrollian stress $\sigma$ gives the Carrollian energy flux, dual to the Galilean matter current $\rho v$. The generalized 
 Carrollian pressure $\varpi$ has units\footnote{The pressure has units of energy density, which in $d=1$ is $\text{energy}\times\text{length}^{-1}$. More accurately, $p$ and $\varpi$ should be called tensions.} of $\text{energy}\times\text{length}^{-1}$, and is mapped to the total energy density $\rho (e + \frac{1}{2} v^2)$.
 
 The remaining Carrollian observables are therefore the total energy density, which is given by $\epsilon + \beta^2 \sigma$, and the momentum density, given by $\beta( \epsilon + \varpi) + \pi$. Here $\epsilon$ is the internal energy density, dual to the Galilean pressure $p$. In full, the system reads (see \cite[Eqs.~(94)--(96)]{AthanasiouPetropoulosSchulzTaujanskas24b})
\begin{equation}\label{eq:carroll_eqns_0}
    \left\lbrace\begin{aligned}
        &\partial_t (\beta \sigma) + \partial_x \sigma = 0, \\ 
        & \partial_t (\epsilon + \beta^2 \sigma) + \partial_x (\beta \sigma) = 0, \\
        & \partial_t (\beta \varpi) + \partial_x \varpi = - \partial_t (\beta \epsilon  + \pi), 
    \end{aligned}\right.
\end{equation}
where $(t,x) \in (0,\infty) \times \mathbb{R} \eqdef \mathbb{R}^2_+$. The first equation in \eqref{eq:carroll_eqns_0} is the so-called Carrollian space-extra equation, and is mapped to the Galilean continuity equation under duality; the second equation in \eqref{eq:carroll_eqns_0} is the Carrollian time-energy equation, and is mapped to the perfect-fluid Galilean space-Euler equation under duality. The third equation is the Carrollian space-momentum equation the Galilean dual of which is the ordinary time-energy equation for ideal fluids.\footnote{Ideal (or equivalently perfect) fluids have neither viscous stress nor heat current. Details on the Carroll--Galilei duality map are available in \cite{AthanasiouPetropoulosSchulzTaujanskas24a,AthanasiouPetropoulosSchulzTaujanskas24b}.} Note that the first two equations decouple from the third, \textit{i.e.}~they do not involve $\varpi$.

We seek a solution $(\sigma, \beta, \varpi)$ in a functional setting to be prescribed, and a priori $\beta$, $\sigma$ and $\varpi$ may assume either sign. Working with dimensionless variables, we set 
\begin{equation} \label{constitutive_relation} \epsilon = \gamma^{-1} \sigma^\gamma. \end{equation}
The constitutive relation \eqref{constitutive_relation} is borrowed directly from the dual Galilean constitutive relation 
$p(\rho) = \gamma^{-1} \rho^\gamma$ for ideal polytropic gases; in particular, $\gamma$ is invariant under the duality map. 
It is a well-known fact in the case of ideal Galilean fluids that the choice $p = \gamma^{-1} \rho^\gamma$, together with the energy equipartition law $e\rho =\frac{p}{\gamma-1}=\frac{\rho^\gamma}{\gamma(\gamma-1)}$ reduces the Galilean time-energy equation, at least in the case of classical solutions, to the simple statement that the entropy is constant in time.\footnote{The dimensionful Galilean constitutive relation reads $p  = K \text{e}^{\nicefrac{s}{c_v}} \rho^\gamma$, where $K$ is a positive constant (involving $\gamma $, the mass of the elementary carriers of the fluid and Planck's constant), $c_v$ is the specific thermal capacity of the gas, and $s $ is its specific entropy. This is how the latter enters the fluid equations.} In this case the Galilean dual of the system \eqref{eq:carroll_eqns_0} therefore reduces to a $2 \times 2$ system of equations known as the isentropic Euler equations. Analogously, using the duality map presented in \cite{AthanasiouPetropoulosSchulzTaujanskas24a}, direct computation shows that if $\sigma$ and $\beta$ are assumed to be classical solutions of the first two equations in \eqref{eq:carroll_eqns_0} with the constitutive relation \eqref{constitutive_relation} 
supplemented by the Carrollian dual of the total energy $\rho\left(e+\frac{1}{2}v^2\right)$ equipartition law,
\begin{equation*}
    \varpi = \frac{1}{2}\sigma\beta^2 + \frac{1}{\gamma(\gamma-1)}\sigma^\gamma, 
\end{equation*}
the triple $(\sigma,\beta,\varpi)$ satisfies the full system \eqref{eq:carroll_eqns_0} with $\pi=0$. In turn, the third equation in \eqref{eq:carroll_eqns_0} is, in some sense, trivialized. We will call this system the \emph{isentropic Carrollian equations}\footnote{The non-isentropic Carrollian equations, \textit{i.e.}~the full system \eqref{eq:carroll_eqns_0}, can of course be studied in their own right. Indeed, if $(\sigma, \beta)$ are known, then the third equation in \eqref{eq:carroll_eqns_0} becomes a scalar conservation law for $\varpi$ with known coefficients. However, our solutions $(\sigma, \beta)$ will be merely $L^\infty$, which is too rough to apply known existence results \cite{BouchutJames1998,Ambrosio} in the large.}---by Galilean analogy since there is no generally accepted definition of Carrollian entropy.
For more details concerning this reduction and the duality between the Carrollian and Galilean systems, we refer the reader to the companion paper \cite[\S 2]{AthanasiouPetropoulosSchulzTaujanskas24b}.

The main target of our analysis is therefore the $2 \times 2$ system of isentropic Carrollian equations \eqref{eq:carroll_eqns_i} for $\sigma$ and $\beta$. In this paper we consider only the case\footnote{In the case $\gamma \neq 3$ it is not clear if one can apply the techniques of compensated compactness; this will be the subject of future work.}
\[ \gamma = 3. \]
For this---and only this---specific value of $\gamma$ the isentropic equations \eqref{eq:carroll_eqns_i} can be recast in the conservative form
\begin{equation}\label{eq:conservative form intro}
    \bu_t + \bF(\bu)_x = 0 
\end{equation}
for the vector of unknowns $\bu = (\sigma,\beta)$ and a suitable flux $\bF$ (see \S\ref{sec:conservative form}). This is the setting in which one may apply compensated compactness, which we use in conjuction with a vanishing viscosity\footnote{Here and throughout the paper we use the term \emph{viscosity} in a strictly PDE sense, to refer to the small parameter $\varepsilon$ multiplying $\bu_{xx}$ not to be confused with the energy density $\epsilon$. Borrowing again from the Galilean terminology, the physical Carrollian viscosity would be instead the coefficient of the first order term in the $\beta$-derivative expansion of the Carrollian stress $\sigma$.} method to construct the pair $(\sigma, \beta)$.

Our main result (\Cref{thm:main_i}) is that for admissible bounded initial data the one-dimensional isentropic Carrollian fluid equations \eqref{eq:carroll_eqns_i} admit global entropy solutions; furthermore the system of equations is weakly continuous with respect to this notion of solution (\Cref{thm:main_ii}). Note that the low regularity is necessary for global existence; in the companion paper\footnote{In \cite{AthanasiouPetropoulosSchulzTaujanskas24b} we in fact study the $C^1$ setting, \emph{without} relying on the conservative form \eqref{eq:conservative form intro} of the equations, for the wider range of exponents $\gamma \in (1, 3]$.} \cite{AthanasiouPetropoulosSchulzTaujanskas24b} we show that solutions which are initially $C^1$ can develop singularities in finite time. We also prove that the problem admits a kinetic formulation (Theorem \ref{thm:main_iii}). We note in passing that, although the entropies of the system have a very simple structure, this is not the case for the entropy-fluxes (see Definition \ref{def:entropy kernels} and \S \ref{sec:entropy pairs}); this complicates the kinetic formulation of the problem, requiring it to be defined in duality with a more intricate set of admissible test functions than is typical.

The proofs of the main results are carried out by obtaining an existence theory and uniform estimates for the vanishing viscosity approximation of the system \eqref{eq:conservative form intro}, and the method of compensated compactness. A number of features distinguish the system \eqref{eq:conservative form intro} from its Galilean counterpart which lead to novel behaviour. Reflective of the duality, the matrix $\nabla_{\bu} \bF(\bu)$ turns out to be the exact inverse of a simple symmetric matrix $M \in \mathbb{R}^{2\times 2}$, and moreover the flux $\bF(\bu)$ manifestly degenerates whenever $|\sigma| = |\beta|$ (see \eqref{eq:flux when gamma is 3}). The eigenvalues of $\nabla_\bu \bF(\bu)$, being reciprocals of the eigenvalues of $M$, never vanish, but may coincide or blow up. The existence of a full set of distinct, well-defined real eigenvalues is precisely the condition for strict hyperbolicity, which can therefore degenerate in one of two ways. The first is the vanishing of Carrollian stress, \emph{i.e.}~the existence of points in phase space where $\sigma = 0$, when the eigenvalues of $M^{-1}$ coincide; this is the dual notion of Galilean cavitation $\rho = 0$, which we propose to call {\emph{Carrollian liquescence}\footnote{In the context of Galilean fluids, the dual notion is $\rho=0$ and is called the vacuum or cavitation. For a Carrollian fluid, $\sigma=0$, $\sigma$ representing stress, suggests no resistance to flow, \emph{i.e.}~the region where the fluid becomes inviscid. We therefore propose to call this phenomenon \emph{Carrollian liquescence}.}. The second, the blow-up of the eigenvalues, occurs when $\beta = \pm \sigma$. In this case the loss of strict hyperbolicity occurs on entire \emph{lines} in phase space, in sharp contrast to the Galilean Euler equations; to our knowledge this phenomenon has no counterpart on the Galilean side since the vanishing of the eigenvalues presents no issue. In fact, $\sigma = 0$ only results in the loss of the \emph{strictness} of hyperbolicity, but $\beta = \pm \sigma$ is a much more severe degeneracy when the eigenvalues cease to even be well-defined and the system loses rank. We therefore define the initial data to be \emph{admissible} if it stays away from both of these degeneracies, and prove that admissibility is propagated. The singularities of the Jacobian $\nabla_\bu \bF(\bu)$ along the lines $\beta = \pm \sigma$ prevent us from directly applying standard results for parabolic systems to deduce the existence of the viscous approximations. This motivates the need for a detailed analysis of the viscous approximate problems, contained in \S \ref{sec:unif est}, where we prove the aforementioned propagation of admissibility using an invariant region argument in the Riemann invariant coordinates; this section contains the bulk of the technical aspects of the paper. The present system \eqref{eq:carroll_eqns_i} therefore exhibits a wealth of new difficulties that are not encountered in the study of its Galilean counterpart, both concerning its hyperbolicity and entropic structure.

The rest of the paper is structured as follows. In \S \ref{sec:prelim} we introduce our notation and rewrite the equations \eqref{eq:carroll_eqns_i} as a first-order conservative system for the unknown vector field $(\sigma,\beta)$, which subsequently allows us to define our concept of entropy solution. In \S \ref{sec:main_results} we provide the statement of our main results. Section \S \ref{sec:hyperbolicity} is concerned with establishing the hyperbolicity and genuine nonlinearity of the system in the relevant regions of phase space; we also introduce the Riemann invariant coordinates in this section. In \S \ref{sec:ent struct and cc}, we compute two linearly independent families of entropy/entropy-flux pairs and establish our compensated compactness framework. In \S \ref{sec:unif est} we prove the existence and uniform boundedness of the viscous approximate solutions, and show that these solutions are admissible in the compensated compactness framework of the previous section. Finally, in \S \ref{sec:proofs of global exis and weak conti}, we use the compensated compactness framework of \S \ref{sec:ent struct and cc} and the construction in \S \ref{sec:unif est} to prove our main theorems.

\section{Preliminaries}\label{sec:prelim}

\subsection{Conventions and Notation}

In the paper we use the notation $\mathbb{R}^2_+ \defeq (0,\infty) \times \mathbb{R}$. The letter $\epsilon$ is used only to denote the internal energy density (dual to the Galilean pressure) in \eqref{eq:carroll_eqns_0}; the letter $\varepsilon$ is used throughout the paper to denote the vanishing viscosity parameter (\textit{cf.}~\S \ref{sec:comp comp framework} and \S \ref{sec:unif est}). We denote by $C^0(X)$ the space of continuous real-valued functions on $X \subset \mathbb{R}^m$, and by $C^k(X)$ the usual spaces of $k$ times continuously differentiable functions on $X$. We denote by $C^\infty(X)$ the space of smooth functions, and use the subscript ${}_c$ to denote spaces of functions with compact support. We write $L^p(X)$ for $p \in [1,\infty]$ to denote the usual Lebesgue spaces on $X$, and $H^s(X)$ and $W^{s,p}(X)$ for Sobolev spaces, as standard. We also employ the following spaces of test functions 
\begin{align}
\mathcal{A} \defeq \big\{ f \in C^2(\mathbb{R}): \, f'(0) = 0 \big\} \quad \text{and} \quad \mathcal{B} \defeq \big\{ g \in C^1(\mathbb{R}): \, g(0) = 0 \big\}. \label{eq:admissible test functions} 
    \end{align}
    The topological dual of any function space $\mathcal{X}$ will be denoted by $\mathcal{X}'$. We denote the vector field of unknowns by $\bu = (\sigma,\beta)$, and we will make use of the following functions of the phase space variables 
\begin{equation*}
    w_1(\sigma,\beta) = \sigma + \beta, \qquad w_2(\sigma,\beta) = \sigma - \beta; 
\end{equation*}
we will show in \S \ref{sec:riemann invariants} that these are Riemann invariants of the system. In \S \ref{sec:comp comp framework}, given a measure $\nu_{t,x}$ on $\mathbb{R}^2$ indexed by $(t,x) \in \mathbb{R}^2_+$, we will employ the notation 
\begin{equation*}
    \overline{h} = \int_{\mathbb{R}^2} h \d \nu_{t,x} ,
\end{equation*}
and thereby drop the subscripts $t$ and $x$, where no confusion arises. A distributional solution will always refer to a solution in duality with $C^\infty_c(\mathbb{R}^2_+)$. We say that a function is a classical solution if it satisfies the equation as a pointwise equality between continuous functions.

\subsection{Conservative Form of the Equations}\label{sec:conservative form}

In this subsection we rewrite the system \eqref{eq:carroll_eqns_i} in the form \eqref{eq:conservative form intro}. A formal application of the chain rule shows that the system \eqref{eq:carroll_eqns_i} reads 
\begin{equation*}
    \left( \begin{matrix}
        \beta & \sigma \\ 
        \sigma^{\gamma-1} + \beta^2 & 2\sigma \beta
    \end{matrix} \right) \partial_t \left( \begin{matrix}
        \sigma \\ 
        \beta
    \end{matrix} \right) + \left( \begin{matrix}
        1 & 0 \\ 
        \beta & \sigma
    \end{matrix} \right) \partial_x \left( \begin{matrix}
        \sigma \\ 
        \beta
    \end{matrix} \right) = 0, 
\end{equation*}
\textit{i.e.} by letting $\bu = (\sigma,\beta)$, 
\begin{equation*}
    \bu_t + \frac{1}{\beta^2 - \sigma^{\gamma-1}}\left( \begin{matrix}
        \beta & -\sigma \\ 
        -\sigma^{\gamma-2}  & \beta
    \end{matrix} \right)  \bu_x = 0, 
\end{equation*}
or equivalently, 
\begin{equation*}
    \bu_t + M^{-1}  \bu_x = 0, 
\end{equation*}
where 
\begin{equation*}
    M = \left(\begin{matrix} \beta & \sigma \\ \sigma^{\gamma-2} & \beta \end{matrix} \right). 
\end{equation*}
In order to implement the compensated compactness method, we rewrite the equations in conservative form, \textit{i.e.} we look for $\bF$ such that 
\begin{equation}\label{eq:what it needs to be}
    \nabla_\bu \bF(\bu) = M^{-1} = \frac{1}{\beta^2 - \sigma^{\gamma-1}}\left( \begin{matrix}
        \beta & -\sigma \\ 
        -\sigma^{\gamma-2}  & \beta
    \end{matrix} \right), 
\end{equation}
whence the system of equations may be rewritten as
\begin{equation}\label{eq:consvn law carroll gamma 3}
    \bu_t + \bF(\bu)_x = 0. 
\end{equation}
It is straightforward to check the equality of the mixed partial derivatives $\mathbf{F}_{\sigma\beta} = \mathbf{F}_{\beta\sigma}$ if and only if $\gamma = 3$, so that the system \eqref{eq:carroll_eqns_i} is in conservative form only for this exponent, as claimed in the introduction. In this case one has 
\begin{equation}\label{eq:flux when gamma is 3}
    \bF(\bu) = \left( \begin{matrix}
        \phi(\sigma/\beta) \\ 
        \frac{1}{2}\log | \beta^2 - \sigma^2 | 
    \end{matrix} \right), 
\end{equation}
with 
\begin{equation*}
    \phi(s) = \left\lbrace \begin{aligned}
        &\artanh s \quad && \text{for } |s| < 1, \\ 
        &\arcoth s \quad && \text{for } |s| > 1. 
    \end{aligned} \right. 
\end{equation*}
Indeed, direct computation yields 
\begin{equation*}
    \bF(\bu)_x = \left( \begin{matrix}
        \frac{\beta \sigma_x}{\beta^2-\sigma^2} - \frac{\sigma \beta_x}{\beta^2-\sigma^2} \\ 
     \frac{\beta \beta_x}{\beta^2-\sigma^2} - \frac{\sigma \sigma_x}{\beta^2-\sigma^2}
    \end{matrix} \right) = \frac{1}{\beta^2-\sigma^2}\left( \begin{matrix}
        \beta & -\sigma \\ 
        -\sigma & \beta
    \end{matrix} \right)\left( \begin{matrix}
        \sigma_x \\ \beta_x
    \end{matrix} \right) = M^{-1}\bu_x, 
\end{equation*}
provided $\beta \neq \pm\sigma$, and the system may be rewritten in the conservative form \eqref{eq:consvn law carroll gamma 3} with $\bF$ as prescribed by \eqref{eq:flux when gamma is 3}. We note in passing that $\bF$ is continuous in the open set 
\[ \mathbb{R}^2 \setminus \{\beta = \pm \sigma\}.\]

\begin{remark}
We note that the system \eqref{eq:consvn law carroll gamma 3} is equivalent to the original equations \eqref{eq:carroll_eqns_i} only for classical solutions, for which applications of the chain rule are justified; this is \emph{a priori} not the case for distributional solutions, which incorporate shocks and other losses of regularity. The reformulation \eqref{eq:consvn law carroll gamma 3} is particularly useful because it is in conservative form for the unknown vector field $(\sigma,\beta)$. This allows us to employ the compensated compactness method introduced by DiPerna in \cite{DiPerna83a} and thereby establish existence in $L^\infty$ without requiring information on the derivatives. 
\end{remark}

\subsection{Entropy Solutions} 

In what follows we consider a stronger notion of solutions than distributional solutions, called entropy solutions. For the time being, we recall that a pair $(\eta,q)$ of $C^2$ functions of the phase-space variables forms an \emph{entropy pair} of the conservative system \eqref{eq:consvn law carroll gamma 3} if there holds 
\begin{equation*}
    \eta(\bu)_t + q(\bu)_x \leq 0 
\end{equation*}
in the sense of distributions; for $\bu$ smooth, the inequality becomes an equality. If $(\eta,q)$ is an entropy pair and $\eta$ is (strongly) convex, then we say that the pair $(\eta,q)$ is a \emph{(strongly) convex entropy pair}. For example, the pair of functions $(\eta^*,q^*)$ given by 
\begin{equation}\label{eq:special entropy}
    \eta^*(\sigma,\beta) = \frac{1}{2}(\sigma^2 + \beta^2), \quad q^*(\sigma,\beta) = \beta 
\end{equation}
is a $C^2$ strongly convex entropy pair for our system; \textit{cf.}~\S \ref{sec:entropy pairs}.

\begin{definition}[Entropy Solution]\label{def:ent sol}
    We say that $\bu=(\sigma,\beta) \in L^\infty(\mathbb{R}^2_+)$ with $\sigma \geq 0$ a.e.~is an \emph{entropy solution} of the system \eqref{eq:carroll_eqns_i} for $\gamma=3$ (\textit{cf.}~\eqref{eq:consvn law carroll gamma 3}) with initial data $\bu_0=(\sigma_0,\beta_0) \in L^\infty(\mathbb{R})$ if$:$ 
    \begin{enumerate}
        \item[(i)] $(\sigma,\beta)$ is a distributional solution of \eqref{eq:consvn law carroll gamma 3} with initial data $\bu_0$, \textit{i.e.}, for all $\varphi \in C^\infty_c([0,\infty)\times\mathbb{R})$, there holds 
        \begin{equation*}
            \int_{\mathbb{R}^2_+} \big( \bu \varphi_t + \bF(\bu) \varphi_x \big) \d x \d t = -\int_\mathbb{R} \bu_0(x) \varphi(0,x) \d x; 
        \end{equation*}
        \item[(ii)] for all $C^2$ convex entropy pairs $(\eta,q)$, there holds 
        \begin{equation}\label{eq:entropy inequality}
    \eta(\bu)_t + q(\bu)_x \leq 0 
\end{equation}
in the sense of distributions, \textit{i.e.} for all non-negative $\varphi \in C^\infty_c(\mathbb{R}^2_+)$ 
\begin{equation*}
    \int_{\mathbb{R}^2_+} \big( \eta(\bu) \varphi_t + q(\bu) \varphi_x  \big) \d x \d t \geq 0. 
\end{equation*}
    \end{enumerate}
\end{definition}

Furthermore, we will show in \S \ref{sec:ent struct and cc} that we may generate entropy pairs $(\eta,q)$ of \eqref{eq:consvn law carroll gamma 3} using two linearly independent sets of kernels, $(\chi_1,\varsigma_1)$ and $(\chi_2,\varsigma_2)$, which we now define. We shall also use these to prove the kinetic formulation of the problem. 
\begin{definition}[Entropy Kernels]\label{def:entropy kernels}
  We define the \emph{first entropy kernel and its entropy-flux kernel} $(\chi,\varsigma_1)$ to be 
  \begin{equation}\label{eq:kernel 1}
    \begin{aligned}
    &\chi_1(\sigma,\beta,s) \defeq \delta(\beta-\sigma-s) + \delta(\beta+\sigma-s), \\ 
    &\varsigma_1(\sigma,\beta,s) \defeq \mathds{1}_{[\beta-\sigma,\beta+\sigma]}(s) \sgn(s-\beta)\frac{1}{s} \frac{\der}{\der s} + 2 \mathds{1}_{[0,\beta]}(s) \frac{1}{s} \frac{\der}{\der s}, 
    \end{aligned}
\end{equation}
and the \emph{second entropy kernel and its entropy-flux kernel} $(\chi_2,\varsigma_2)$ to be 
  \begin{equation}\label{eq:kernel 2}
    \begin{aligned}
     &\chi_2(\sigma,\beta,s) \defeq \mathds{1}_{[\beta-\sigma,\beta+\sigma]}(s), \\ & \varsigma_2(\sigma,\beta,s) \defeq \mathds{1}_{[\beta-\sigma,\beta+\sigma]}(s)\frac{1}{s}. 
    \end{aligned}
\end{equation}
\end{definition}
For a suitable test function $\varphi$, we will generate entropy pairs $(\eta_j^\varphi,q_j^\varphi)$ ($j=1,2$) using the formulas 
\begin{equation}\label{eq:generate my entropies}
    \eta_j^\varphi(\sigma,\beta) = \langle \chi_j(\sigma,\beta,\cdot), \varphi \rangle, \qquad q_j^\varphi(\sigma,\beta) = \langle \varsigma_j(\sigma,\beta,\cdot) , \varphi \rangle; 
\end{equation}
we refer the reader to \S \ref{sec:entropy pairs} for further details.

\section{Main Results} \label{sec:main_results}

Our first main theorem is as follows, and is obtained as a consequence of the compensated compactness framework established in \S \ref{sec:comp comp framework} (\textit{cf.}~Proposition \ref{lem:comp comp}).

\begin{theorem}[Global Existence]\label{thm:main_i}
        Let $c_0>0$ and $(\sigma_0,\beta_0) \in L^2(\mathbb{R}) \cap L^\infty(\mathbb{R})$ be such that 
        \begin{equation}\label{eq:initial condition essinf}
           \essinf_\mathbb{R} \big( \sigma_0 - |\beta_0| \big) \geq c_0. 
        \end{equation}
        Then there exists $(\sigma,\beta) \in L^\infty(\mathbb{R}^2_+) \cap L^\infty_{\loc}(0,\infty;L^2(\mathbb{R}))$ an entropy solution of \eqref{eq:carroll_eqns_i} for $\gamma=3$ with initial data $(\sigma_0,\beta_0)$. Moreover, 
         \begin{equation*}
           \essinf_{\mathbb{R}^2_+} \big( \sigma - |\beta| \big) \geq c_0. 
        \end{equation*}
\end{theorem}

We emphasise that $c_0>0$ in \eqref{eq:initial condition essinf} is fixed arbitrarily in the statement of Theorem \ref{thm:main_i} and does not change throughout the rest of the paper; this condition is of course equivalent to 
\begin{equation*}
    \essinf_\mathbb{R} \big( \sigma_0 - |\beta_0| \big) > 0, 
\end{equation*}
however we prefer this formulation to simplify notation in \S \ref{sec:unif est} and to be uniform in our statement of condition \eqref{eq:ic essinf second} in Theorem \ref{thm:main_ii}. While the requirement $(\sigma_0,\beta_0) \in L^2(\mathbb{R})$ may seem excessive at first sight, this is a natural assumption given the special entropy pair \eqref{eq:special entropy}, which is in a sense the most natural underlying strongly convex entropy of the system.

Our second main theorem is concerned with the weak continuity of the system of equations, which shows the weak compactness of entropy solutions of \eqref{eq:carroll_eqns_i} for $\gamma=3$. Once again, this result is obtained as a consequence of the compensated compactness framework established in \S \ref{sec:comp comp framework}; Proposition \ref{lem:comp comp}.

\begin{theorem}[Weak Continuity]\label{thm:main_ii}
    Let $c_0>0$ and $\{(\sigma^n,\beta^n)\}_n \subset L^2(\mathbb{R}) \cap L^\infty(\mathbb{R})$ be a sequence of entropy solutions of \eqref{eq:carroll_eqns_i} with $\gamma=3$, uniformly bounded in $L^\infty(\mathbb{R}^2_+)$ and satisfying 
    \begin{equation}\label{eq:ic essinf second}
           \essinf_{\mathbb{R}^2_+} \big( \sigma^n - |\beta^n| \big) \geq c_0 \quad \text{for all }\ n \in \mathbb{N}. 
        \end{equation}
        Then there exists $(\sigma,\beta) \in L^\infty(\mathbb{R}^2_+)$ and a subsequence, still labelled $\{(\sigma^n,\beta^n)\}_n$, such that 
        \begin{equation*}
            (\sigma^n,\beta^n) \longrightarrow (\sigma,\beta) \quad \text{a.e.~in }\ \mathbb{R}^2_+, 
        \end{equation*}
        and furthermore $(\sigma,\beta)$ is an entropy solution of \eqref{eq:carroll_eqns_i} with $\gamma=3$. 
\end{theorem}

Our third main theorem quantifies the entropy dissipation \eqref{eq:entropy inequality} in our low regularity framework; it is called the \emph{kinetic formulation} of the system of conservation laws \eqref{eq:consvn law carroll gamma 3} (\textit{cf.}~\cite{PerthameBook}), and makes use of the entropy kernels introduced in Definition \ref{def:entropy kernels} and the function spaces $\mathcal{A}$ and $\mathcal{B}$ introduced in \eqref{eq:admissible test functions}. This theorem establishes a PDE version of the second law of thermodynamics for the isentropic Carrollian fluid system.

\begin{theorem}[Kinetic Formulation]\label{thm:main_iii}
    Let $(\sigma,\beta)$ be an entropy solution obtained in Theorem \ref{thm:main_i}. Then there exists $\mu_1$ a non-negative bounded Radon measure on $\mathbb{R}^2_+\times\mathbb{R}$, and $\mu_2$ a signed bounded Radon measure on $\mathbb{R}^2_+\times\mathbb{R}$ such that 
    \begin{equation}\label{eq:kinetic formulation 1}
        \partial_t \big(\chi_1(\sigma(t,x),\beta(t,x),s)\big) + \partial_x \big(\varsigma_1(\sigma(t,x),\beta(t,x),s)\big) = -\partial^2_s \mu_1(t,x,s) \quad \text{in } \big( C^2_c(\mathbb{R}^2_+)\times\mathcal{A} \big)', 
    \end{equation}
    and 
    \begin{equation}\label{eq:kinetic formulation 2}
        \partial_t \big(\chi_2(\sigma(t,x),\beta(t,x),s)\big) + \partial_x \big(\varsigma_2(\sigma(t,x),\beta(t,x),s)\big) = -\partial_s \mu_2(t,x,s) \quad \text{in } \big( C^2_c(\mathbb{R}^2_+)\times\mathcal{B} \big)'. 
        \end{equation}
\end{theorem}

\section{Hyperbolicity, Genuine Nonlinearity, and Riemann Invariants}\label{sec:hyperbolicity}

In this section, we establish the hyperbolicity and genuine nonlinearity of the system; both of which are essential ingredients to apply the compensated compactness method (\textit{cf.}~\S \ref{sec:comp comp framework}). We also compute the Riemann invariants of the systems, which shall subsequently be used in \S \ref{sec:unif est} to compute the invariant regions of the phase space to which solutions are constrained. 

\subsection{Hyperbolicity} 

\begin{lemma}[Strict hyperbolicity]\label{lem:strictly hyp}
  For $\gamma=3$, the system \eqref{eq:consvn law carroll gamma 3} is strictly hyperbolic in the subregion 
    \begin{equation*}
        \mathcal{H} \defeq \Big\{ (\sigma,\beta) \in \mathbb{R}^2 : \, \sigma \neq 0, \, \beta \neq \pm \sigma  \Big\}. 
    \end{equation*}
\end{lemma}

\begin{proof}
The eigenvalues of $M$ are given by 
\begin{equation*}
    \lambda_1' = \beta - \sigma, \quad \lambda_2' = \beta + \sigma, 
\end{equation*}
whence the eigenvalues of $M^{-1}$ are given by 
\begin{equation*}
    \lambda_1 = \frac{1}{\beta - \sigma}, \quad \lambda_2 = \frac{1}{\beta + \sigma}, 
\end{equation*}
and thus the system is strictly hyperbolic except along the curves $\{\sigma = 0\}\cup \{\beta = \pm \sigma\}$; the eigenvalues are real and distinct on $\mathcal{H}$.
\end{proof}

Furthermore, a set of normalised right eigenvectors of $M$ is provided by 
\begin{equation*}
    \br_1 = \frac{1}{\sqrt{2}}\left( \begin{matrix}
        1 \\ 
        -1
    \end{matrix} \right)    , \quad \br_2 =  \frac{1}{\sqrt{2}}\left( \begin{matrix}
        1 \\ 
        1
    \end{matrix} \right); 
\end{equation*}
they are also right eigenvectors for $M^{-1}$, with eigenvalues given by $\lambda_1,\lambda_2$ respectively. In turn, we have the diagonalised form: 
\begin{equation}\label{eq:diagonalised jacobian F}
    \nabla_\bu \bF(\bu) = \frac{1}{2}\left( \begin{matrix}
        1 & 1 \\ 
        -1 & 1
    \end{matrix} \right) \left( \begin{matrix}
        \frac{1}{\beta-\sigma} & 0 \\ 
        0 & \frac{1}{\beta+\sigma} 
    \end{matrix} \right) \left( \begin{matrix}
        1 & -1 \\ 
        1 & 1
    \end{matrix} \right). 
\end{equation}

\subsection{Genuine Nonlinearity} 

\begin{lemma}[Genuine Nonlinearity]\label{lem:genuinely nonlinear}
    For $\gamma=3$, the system \eqref{eq:consvn law carroll gamma 3} is genuinely nonlinear, including along the curves $\{\sigma=0\}$ and $\{\beta = \pm \sigma\}$. 
\end{lemma}

\begin{proof}
We verify the condition of genuine nonlinearity, \textit{i.e.}
\begin{equation*}
    \nabla_\bu \lambda_j \cdot \br_j \neq 0 
\end{equation*}
for $j \in \{ 1, 2 \}$. To this end, we compute 
\begin{equation*}
    \nabla_\bu \lambda_1 = \frac{1}{(\beta-\sigma)^2}\left(  \begin{matrix}
        1 \\ 
        -1
    \end{matrix}\right), \qquad \nabla_\bu \lambda_2 = -\frac{1}{(\beta+\sigma)^2}\left(  \begin{matrix}
        1 \\ 
        1
    \end{matrix}\right), 
\end{equation*}
whence 
\begin{equation*}
    \nabla_\bu \lambda_1 \cdot \br_1 = \frac{\sqrt{2}}{(\beta-\sigma)^2}, \qquad \nabla_\bu \lambda_2 \cdot \br_2 = -\frac{\sqrt{2}}{(\beta+\sigma)^2}, 
\end{equation*}
and genuine nonlinearity follows. 
\end{proof}

\subsection{Riemann Invariants}\label{sec:riemann invariants} 

\begin{lemma}[Riemann Invariants]\label{lem:riemann invariants}
    The functions 
    \begin{equation}\label{eq:riemann invariants}
    w_1  \defeq  \sigma+\beta    , \qquad    w_2  \defeq   \sigma-\beta   , 
\end{equation}
are Riemann invariants of the system \eqref{eq:consvn law carroll gamma 3} for $\gamma=3$. Furthermore, the mapping $(\sigma,\beta) \mapsto (w_1,w_2)$ is bijective. 
\end{lemma}

\begin{proof}
Direct computations show that for $j \in \{1, 2\}$
\begin{equation*}
    \nabla_{\bu}w_j \cdot \br_j = 0,
\end{equation*}
\emph{i.e.}~that $w_1,w_2$ is a set of Riemann invariants for the system. Furthermore, the curvilinear coordinate transformation 
\begin{equation*}
    (\sigma,\beta) \longmapsto (w_1,w_2) 
\end{equation*}
is manifestly bijective on the full phase plane of solutions. 
\end{proof}

We will see in \S \ref{sec:unif est} that the Riemann invariants will allow us to obtain a priori $L^\infty$ estimates on $\sigma$ and $\beta$. Note also that 
\begin{equation*}
    \nabla_\bu w_1 = \left(  \begin{matrix}
        1 \\ 1
    \end{matrix}\right) = \sqrt{2}\br_2, \qquad \nabla_\bu w_2 = \left( \begin{matrix}
        1 \\ -1
    \end{matrix} \right) = \sqrt{2}\br_1. 
\end{equation*}

\section{Entropy Structure and Compensated Compactness}\label{sec:ent struct and cc}

In this section, we compute two linearly independent families of entropy pairs of the system \eqref{eq:carroll_eqns_i} in \S \ref{sec:entropy pairs}, and then proceed to apply the compensated compactness method to obtain a general compactness framework in \ref{sec:comp comp framework}. This compactness framework will subsequently be used in \S \ref{sec:proofs of global exis and weak conti} to prove both the existence of a global entropy solution and the weak continuity of the system \eqref{eq:carroll_eqns_i}. 

\subsection{Entropy Structure}\label{sec:entropy pairs}

We seek pairs of functions $(\eta,q)$ for which one can obtain additional conservation laws from \eqref{eq:consvn law carroll gamma 3}. To this end, let $\eta \in C^2(\mathbb{R})$ and pre-multiply the system \eqref{eq:consvn law carroll gamma 3} by $\nabla_\bu \eta^\intercal$. Formally, \textit{i.e.}~assuming smooth $\bu$, we get 
\begin{equation*}
    \underbrace{\nabla_\bu \eta ^\intercal \bu_t}_{=\eta(\bu)_t} + \nabla_\bu \eta^\intercal \nabla_\bu \bF(\bu) \bu_x = 0. 
\end{equation*}
We seek to find a corresponding entropy-flux $q$ such that 
\begin{equation*}
    \nabla_\bu q^\intercal = \nabla_\bu \eta^\intercal \nabla_\bu \bF(\bu), 
\end{equation*}
\textit{i.e.}
\begin{equation*}
    q_\sigma =  \frac{\beta \eta_\sigma - \sigma \eta_\beta}{\beta^2-\sigma^2}  , \qquad q_\beta = \frac{\beta \eta_\beta - \sigma \eta_\sigma}{\beta^2-\sigma^2}       . 
\end{equation*}
Using the equality of the mixed partial derivatives $q_{\beta \sigma} = q_{\sigma \beta}$, we obtain the \emph{entropy equation} 
\begin{equation}\label{eq:ent eqn}
    \sigma(\beta^2-\sigma^2)^2 (\eta_{\sigma\sigma} - \eta_{\beta\beta}) = 0. 
\end{equation}
In turn, all solutions of the linear wave equation 
\begin{equation*}
    \eta_{\sigma\sigma} - \eta_{\beta \beta} = 0 
\end{equation*}
are entropies for our system. Recall that the general solution of the above is given by the usual d'Alembert representation formula: 
\begin{equation}\label{eq:eta general}
    \eta(\sigma,\beta) = f(\beta-\sigma) + f(\beta+\sigma) + \int_{\beta-\sigma}^{\beta+\sigma} g(s) \d s \quad (\sigma>0),
\end{equation}
where $f = \frac{1}{2}\eta|_{\sigma=0}$, $g = \frac{1}{2}\eta_\sigma|_{\sigma=0}$; by varying the choice of $f,g$, we generate \emph{two linearly independent families} of solutions. 

To compute the corresponding entropy-flux $q$, we have 
\begin{equation*}
    \begin{aligned} 
    &\eta_\sigma = f'(\beta+\sigma) - f'(\beta-\sigma) + g(\beta+\sigma) + g(\beta-\sigma), \\ 
    & \eta_\beta = f'(\beta+\sigma) + f'(\beta-\sigma) + g(\beta+\sigma) - g(\beta-\sigma), 
    \end{aligned}
\end{equation*}
and thus 
\begin{equation*}
   \begin{aligned} 
   q_\sigma &= \frac{1}{\beta^2-\sigma^2}( \beta \eta_\sigma - \sigma \eta_\beta ) \\ 
   &= \frac{f'(\beta+\sigma)}{\beta+\sigma} - \frac{f'(\beta-\sigma)}{\beta-\sigma} + \frac{1}{\beta+\sigma}g(\beta+\sigma) + \frac{1}{\beta-\sigma}g(\beta-\sigma),
   \end{aligned}
\end{equation*}
while 
\begin{equation*}
    \begin{aligned}
        q_\beta &= \frac{1}{\beta^2-\sigma^2}(\beta \eta_\beta - \sigma \eta_\sigma) \\ 
        &= \frac{f'(\beta+\sigma)}{\beta+\sigma} + \frac{f'(\beta-\sigma)}{\beta-\sigma} + \frac{1}{\beta+\sigma}g(\beta+\sigma) - \frac{1}{\beta-\sigma}g(\beta-\sigma).
    \end{aligned}
\end{equation*}
It follows that 
\begin{equation}\label{eq:q general form}
    q(\sigma,\beta) = \int_{\beta}^{\beta+\sigma} \frac{f'(s)}{s} \d s -  \int_{\beta-\sigma}^{\beta} \frac{f'(s)}{s} \d s + 2 \int_{0}^\beta \frac{f'(s)}{s} \d s + \int_{\beta-\sigma}^{\beta+\sigma} \frac{g(s)}{s} \d s 
\end{equation}
is a suitable entropy-flux for $\eta$, provided these integrals are well-defined. Recalling Definition \ref{def:entropy kernels}, we have shown that a generic entropy pair may be written as claimed in \eqref{eq:generate my entropies}, \textit{i.e.} 
\begin{align}
 &\eta(\sigma,\beta) = \langle\chi_1(\sigma,\beta,\cdot),f\rangle + \langle\chi_2(\sigma,\beta,\cdot),g\rangle, \quad 
 &q(\sigma,\beta) = \langle\varsigma_1(\sigma,\beta,\cdot),f\rangle + \langle\varsigma_2(\sigma,\beta,\cdot),g\rangle. \label{eq:kernel exp ent pair}
\end{align}
We have the following result concerning the admissible test functions for the kernels, for which we recall for convenience the spaces of functions $\mathcal{A}$ and $\mathcal{B}$ introduced in \eqref{eq:admissible test functions}: 
\begin{align}
\mathcal{A} = \big\{ f \in C^2(\mathbb{R}): \, f'(0) = 0 \big\} \quad \text{and} \quad \mathcal{B} = \big\{ g \in C^1(\mathbb{R}): \, g(0) = 0 \big\}. \nonumber
    \end{align}

\begin{lemma}\label{lem:admissibility kernels}
    The kernels $\chi_1$ and $\chi_2$ are well-defined on $C^0(\mathbb{R})$, while $\varsigma_1$ is well-defined on $\mathcal{A}$ and $\varsigma_2$ is well-defined on 
    $\mathcal{B}$. 
\end{lemma}

\begin{proof}
Note that $\chi_1$ and $\chi_2$ are well-defined on elements of $C^0(\mathbb{R})$. On the other hand, for all $g \in C^1(\mathbb{R})$, if $\beta+\sigma>0$ and $\beta-\sigma=0$, 
\begin{equation*}
    |\langle \varsigma_2(\sigma,\beta,\cdot), g \rangle| = \bigg|\int^{\beta+\sigma}_{0} \frac{g(s)}{s} \d s \bigg| 
\end{equation*}
and the above is well-defined if and only if $g(0)=0$; this can be seen from an application of Taylor's Theorem. If $\beta+\sigma>0$ and $\beta-\sigma<0$, then one may rewrite the duality product as a principal value plus a remainder term: 
\begin{equation*}
    \begin{aligned}
        |\langle \varsigma_2(\sigma,\beta,\cdot), g \rangle| &\leq \bigg|\int^{\beta+\sigma}_{\sigma-\beta} \frac{g(s)}{s} \d s \bigg| + \bigg|\int^{\sigma-\beta}_{-(\sigma-\beta)} \frac{g(s)}{s} \d s \bigg| \\ 
        &\leq \min\{\beta+\sigma,\sigma-\beta\}^{-1} \Vert g \Vert_{L^\infty([\sigma-\beta,\beta+\sigma])} + \int^{\sigma-\beta}_{0} \frac{|g(s)-g(-s)|}{s} \d s, 
    \end{aligned}
\end{equation*}
and one estimates the above right-hand side using Taylor's Theorem. It is also easy to verify that the duality product is well-defined whenever both $\beta+\sigma > 0$ and $\beta-\sigma>0$. We deduce that $\varsigma_2$ is well-defined on elements $g \in C^1(\mathbb{R})$ such that $g(0) = 0$. An analogous argument shows that $\varsigma_1$ is well-defined on elements $f \in C^2(\mathbb{R})$ such that $f'(0)=0$. 
\end{proof}

We therefore have the following lemma.
\begin{lemma}[Entropy pairs]\label{lem:ent pairs}
    Let $(f,g)$ be any pair of functions in $\mathcal{A} \times \mathcal{B}$. Then the pairs of functions $(\eta_1,q_1)$ and $(\eta_2,q_2)$ defined by 
    \begin{align}
    &\eta_1(\sigma,\beta) = f(\beta-\sigma) + f(\beta+\sigma) , \qquad q_1(\sigma,\beta) = \int_{\beta-\sigma}^{\beta+\sigma} \sgn(s-\beta) \frac{f'(s)}{s} \d s + 2 \int_{0}^\beta \frac{f'(s)}{s} \d s, \nonumber
    \end{align}
    and 
  \begin{align}
    &\eta_2(\sigma,\beta) = \int_{\beta-\sigma}^{\beta+\sigma} g(s) \d s, \qquad q_2(\sigma,\beta) = \int_{\beta-\sigma}^{\beta+\sigma} \frac{g(s)}{s} \d s, \nonumber
    \end{align}
    are entropy pairs for the system \eqref{eq:consvn law carroll gamma 3}; moreover, in the region $\{\beta \neq \pm \sigma\}$, they are $C^2$. In particular, the choice $(f(s),g(s)) = (\frac{1}{4}s^2,0) \in \mathcal{A} \times \mathcal{B}$ yields the special entropy pair $(\eta^*,q^*) = (\frac{1}{2}(\sigma^2 + \beta^2), \beta)$ of \eqref{eq:special entropy}.
\end{lemma}

\begin{proof}
Let $(f,g) \in \mathcal{A} \times \mathcal{B}$ as in the statement. Then by \Cref{lem:admissibility kernels} the pairs of functions $(\eta_j,q_j)$ ($j=1,2$) in the statement of \Cref{lem:ent pairs} are well-defined. Furthermore, they are entropy pairs in view of \eqref{eq:kernel exp ent pair} and the computations preceding \eqref{eq:kernel exp ent pair}.
It remains to verify their regularity. We compute the first derivatives 
\begin{equation*}
   \begin{aligned} &\nabla_\bu \eta_1 = f'(\beta-\sigma) \left( \begin{matrix}-1 \\ 1 \end{matrix} \right) + f'(\beta+\sigma) \left( \begin{matrix}
        1 \\ 1
    \end{matrix} \right) , \quad && \nabla_\bu q_1 = \frac{f'(\beta-\sigma)}{\beta-\sigma} \left( \begin{matrix}-1 \\ 1 \end{matrix} \right) + \frac{f'(\beta+\sigma)}{\beta+\sigma} \left( \begin{matrix}
        1 \\ 1
    \end{matrix} \right), \\ 
    &\nabla_\bu \eta_2 = g(\beta-\sigma) \left( \begin{matrix} 1 \\ -1 \end{matrix} \right) + g(\beta+\sigma) \left( \begin{matrix}
        1 \\ 1
    \end{matrix} \right) , \quad && \nabla_\bu q_2 = \frac{g(\beta-\sigma)}{\beta-\sigma} \left( \begin{matrix} 1 \\ -1 \end{matrix} \right) + \frac{g(\beta+\sigma)}{\beta+\sigma} \left( \begin{matrix}
        1 \\ 1
    \end{matrix} \right), 
    \end{aligned}
\end{equation*}
along with the second derivatives 
\begin{equation*}
   \begin{aligned}
       &\nabla_\bu^2 \eta_1 = f''(\beta-\sigma) \left( \begin{matrix}
        1 & -1 \\ -1 & 1
    \end{matrix} \right) + f''(\beta+\sigma) \left( \begin{matrix}
        1 & 1 \\ 1 & 1
    \end{matrix} \right), \quad \\ 
    &\nabla^2_\bu q_1 \!\! = \! \frac{f''(\beta-\sigma)}{\beta-\sigma}\!\! \left( \begin{matrix}
        1 & -1 \\ 
        -1 & 1
    \end{matrix} \right) \!\! + \! \frac{f''(\beta+\sigma)}{\beta+\sigma} \!\!\left( \begin{matrix}
        1 & 1 \\ 
       1  & 1
    \end{matrix}  \right) \!\!+\! \frac{f'(\beta-\sigma)}{(\beta-\sigma)^2}\!\!\left( \begin{matrix}
        -1 & 1 \\ 
       1  & -1
    \end{matrix}  \right) \!\! +\! \frac{f'(\beta+\sigma)}{(\beta+\sigma)^2} \!\!\left( \begin{matrix}
        -1 & -1 \\ 
       -1  & -1
    \end{matrix}  \right), 
   \end{aligned}
\end{equation*}
and 
\begin{equation*}
    \begin{aligned}
        &\nabla^2_\bu \eta_2 = g'(\beta-\sigma) \left( \begin{matrix}
            -1 & 1 \\ 
            1 & -1
        \end{matrix}\right) + g'(\beta+\sigma) \left( \begin{matrix}
            1 & 1 \\ 
            1 & 1
        \end{matrix}\right), \\ 
        &\nabla^2_\bu q_2 \!\! = \! \frac{g'(\beta-\sigma)}{\beta-\sigma}\!\! \left( \begin{matrix}
        -1 & 1 \\ 
        1 & -1
    \end{matrix} \right) \!\! + \! \frac{g'(\beta+\sigma)}{\beta+\sigma} \!\!\left( \begin{matrix}
        1 & 1 \\ 
       1  & 1
    \end{matrix}  \right) \!\!+\! \frac{g(\beta-\sigma)}{(\beta-\sigma)^2}\!\!\left( \begin{matrix}
        1 & -1 \\ 
      -1  & 1
    \end{matrix}  \right) \!\! +\! \frac{g(\beta+\sigma)}{(\beta+\sigma)^2} \!\!\left( \begin{matrix}
        -1 & -1 \\ 
      -1  & -1
    \end{matrix}  \right). 
    \end{aligned}
\end{equation*}
It is immediate from the formulas above that $(\eta_j,q_j)$ ($j=1,2$) are $C^2$ in the region $\{\beta \neq \pm \sigma\}$. 
\end{proof}

\subsection{Compensated Compactness Framework}\label{sec:comp comp framework}

In this subsection, we establish our compactness framework; this will be used in \S \ref{sec:proofs of global exis and weak conti} in order to pass to the limit in the regularised problem which we study in \S \ref{sec:unif est}, thus obtaining an entropy solution of \eqref{eq:carroll_eqns_i}. 

\begin{prop}[Compensated Compactness Framework]\label{lem:comp comp}
    Let $\{\bu^\varepsilon=(\sigma^\varepsilon,\beta^\varepsilon)\}_\varepsilon$  be a uniformly bounded sequence in $L^\infty(\mathbb{R}^2_+)$ such that$:$ 
    \begin{enumerate}
        \item[(i)] there exists a positive constant $c>0$ such that $\sigma^\varepsilon- |\beta^\varepsilon| \geq c,$ and 
        \item[(ii)] for all entropy pairs $(\eta,q)$ generated via Lemma \ref{lem:ent pairs}, the sequence of entropy dissipation measures 
        $$\big\{\eta(\bu^\varepsilon)_t + q(\bu^\varepsilon)_x \big\}_\varepsilon$$ 
        is pre-compact in $H^{-1}_\loc(\mathbb{R}^2_+),$ \textit{i.e.} for all compact $K \subset \mathbb{R}^2_+$ there exists a compact $\kappa_K \subset H^{-1}(K)$ such that $$\big\{\eta(\bu^\varepsilon)_t + q(\bu^\varepsilon)_x \big\}_\varepsilon \subset \kappa_K.$$
    \end{enumerate}
    Then there exists a subsequence, which we do not relabel, and $\bu \in L^\infty(\mathbb{R}^2_+)$ such that there holds 
    \begin{equation*}
        \bu^\varepsilon \to \bu \quad \text{a.e.~and strongly in } L^p_\loc(\mathbb{R}^2_+) \text{ for all }\, p \in [1,\infty). 
    \end{equation*}
\end{prop}

As per the notations introduced in \S \ref{sec:prelim}, we remind the reader that in what follows, for a measure $\nu_{t,x}$ on $\mathbb{R}^2$ indexed by $(t,x) \in \mathbb{R}^2_+$, we employ the notation 
\begin{equation*}
    \overline{h} = \int_{\mathbb{R}^2} h \d \nu_{t,x} \, ,
\end{equation*}
where no confusion arises regarding the subscript $t,x$.

\begin{proof}
  The uniform boundedness in $L^\infty$ implies, using Alaoglu's Theorem, the existence of a weakly-* convergent subsequence, still labelled $\{\bu^\varepsilon\}_\varepsilon$, such that $\bu^\varepsilon \overset{*}{\rightharpoonup} \bu$ in $L^\infty(\mathbb{R}^2_+)^2$. The remainder of this proof is focused on improving this convergence to convergence almost everywhere.

   \smallskip 
   \noindent 1. \textit{Young measure and commutation relation}. The Fundamental Theorem of Young measures (\textit{cf.}~\cite[\S 2, Theorem 9.1.5]{Serre2}) implies that the sequence $\{\bu^\varepsilon\}_\varepsilon$ generates a family $\{\nu_{t,x}\}_{t,x}$ of probability measures on $\mathbb{R}^2$, for which, for a further subsequence which we do not relabel, there holds 
\begin{equation*}
    h(\bu^\varepsilon) \overset{*}{\rightharpoonup} \overline{h} \quad \text{in } L^\infty(\mathbb{R}^2_+) 
\end{equation*}
for all $h \in C^0(\mathbb{R}^2)$. In turn, using the continuity of the entropy pairs generated by Lemma \ref{lem:ent pairs}, we find that 
 \begin{equation*}
     \eta_i(\bu^\varepsilon) q_j(\bu^\varepsilon) \overset{*}{\rightharpoonup} \overline{\eta_i q_j} \qquad (i \neq j), 
 \end{equation*}
 for all entropy pairs generated by Lemma \ref{lem:ent pairs}. Meanwhile, using the $H^{-1}_\loc$-compactness criterion (assumption (ii) of the lemma), the Div-Curl Lemma implies that
 \begin{equation*}
     \eta_1(\bu^\varepsilon) q_2(\bu^\varepsilon) - \eta_2(\bu^\varepsilon) q_1(\bu^\varepsilon) \to \overline{\eta_1} \, \overline{q_2} - \overline{\eta_2} \, \overline{q_1} \qquad \text{in } \mathcal{D}(\mathbb{R}^2_+). 
 \end{equation*}
We deduce that there holds the commutation relation 
    \begin{equation}\label{eq:commutation relation}
        \overline{\eta_1 q_2 - \eta_2 q_1} = \overline{\eta_1} \, \overline{q_2} - \overline{\eta_2} \, \overline{q_1}
    \end{equation}
    for all entropy pairs $(\eta_1,q_1)$ and $(\eta_2,q_2)$ generated by Lemma \ref{lem:ent pairs}.

    \smallskip 
    \noindent 2. \textit{Young measure reduction}. The system \eqref{eq:consvn law carroll gamma 3} is strictly hyperbolic and genuinely nonlinear by virtue of assumption (i) and Lemmas \ref{lem:strictly hyp} and \ref{lem:genuinely nonlinear}. Furthermore, Lemma \ref{lem:ent pairs} provides us with two linearly independent families of entropy pairs, which are pre-compact in $H^{-1}_\loc$ by assumption (ii) of the proposition. In turn, using \eqref{eq:commutation relation} and the classical arguments of Serre \cite[\S 2]{Serre2} or DiPerna \cite{DiPerna83a}, we deduce that there exists a measurable function $\bu \in L^\infty(\mathbb{R}^2_+)^2$ such that 
    \begin{equation*}
        \nu_{t,x} = \delta_{\bu(t,x)}. 
    \end{equation*}
    It follows that $\bu^\varepsilon \to \bu$ almost everywhere. Using the Dominated Convergence Theorem and the uniform boundedness, we deduce that $\bu^\varepsilon \to \bu$ strongly in $L^p_\loc(\mathbb{R}^2_+)^2$ for all $p \in [1,\infty)$. 
\end{proof}

\section{Regularised Problem and Uniform Estimates}\label{sec:unif est}

We turn now to establishing the existence and uniform estimates for the vanishing viscosity approximation of our system \eqref{eq:conservative form intro}. For $\varepsilon>0$ we consider the initial value problem 
\begin{equation}\label{eq:regularised pb}
    \left\lbrace\begin{aligned}
        &\bu^\varepsilon_t + \bF(\bu^\varepsilon)_x = \varepsilon  \bu^\varepsilon_{xx}, \\ 
        &\bu^\varepsilon|_{t=0} = \bu_0. 
    \end{aligned}\right.
\end{equation}
Ultimately, our objective will be to take the limit as $\varepsilon \to 0$ and show that the sequence $\{\bu^\varepsilon\}_\varepsilon$ converges almost everywhere to $\bu$ an entropy solution of \eqref{eq:conservative form intro}. We emphasise that, in view of the singularities contained in the term $\nabla_{\bu}\bF$, \textit{cf.}~\eqref{eq:what it needs to be}, the existence of solutions to \eqref{eq:regularised pb} cannot be deduced by straightforward application of the general theory for parabolic systems \cite{ladyzhenskaia1988linear}. 

\smallskip 

Our main result in this section is as follows, where we recall that the positive constant $c_0$ was fixed in the statement of Theorem \ref{thm:main_i}.

\begin{prop}[Existence for Regularised Problems]\label{prop:approx pbms}
    Let $\varepsilon>0$, $T>0$, and $\bu_0 \in L^2(\mathbb{R}) \cap L^\infty(\mathbb{R})$ satisfy 
    \begin{equation*}
        \essinf_\mathbb{R} w_1(\bu_0) , ~  \essinf_\mathbb{R} w_2(\bu_0) \geq c_0. 
    \end{equation*}
    Then there exists $\bu^\varepsilon \in L^\infty(\mathbb{R}^2_+) \cap L^2_{\loc}(0,T;H^1(\mathbb{R}))$ with $\bu^\varepsilon_t \in L^2_{\loc}(0,T;H^{-1}(\mathbb{R}))$, smooth away from the initial time, a classical solution of the regularised problem \eqref{eq:regularised pb} on $(0,T)\times\mathbb{R}$, \textit{i.e.} 
    \begin{equation*}
    \left\lbrace\begin{aligned}
        &\bu^\varepsilon_t + \bF(\bu^\varepsilon)_x = \varepsilon  \bu^\varepsilon_{xx}, \\ 
        &\bu^\varepsilon|_{t=0} = \bu_0. 
    \end{aligned}\right.
\end{equation*}
Furthermore, $\bu^\varepsilon$ satisfies 
\begin{equation*}
        w_1(\bu^\varepsilon) , ~ w_2(\bu^\varepsilon) \geq c_0 \quad \text{a.e.~in }\mathbb{R}^2_+, 
    \end{equation*}
    and there exists a positive constant $C$ independent of $\varepsilon$ and $T$ such that 
    \begin{equation*}
        \Vert \bu^\varepsilon \Vert_{L^\infty(\mathbb{R}^2_+)} \leq C. 
    \end{equation*}
\end{prop}

The previous proposition rests on the fact that the system of equations admits invariant regions in phase space; namely, for initial data satisfying $\essinf_\mathbb{R} w_1(\bu_0) , \, \essinf_\mathbb{R} w_2(\bu_0) \geq c_0$, there holds $w_1(\bu^\varepsilon) , \, w_2(\bu^\varepsilon) \geq c_0$ a.e., whence the term $\bF(\bu^\varepsilon)$ is well-defined. This can be seen from the following \emph{formal} argument: by multiplying \eqref{eq:regularised pb} with $\nabla_{\bu} w_1^\intercal$, we obtain 
\begin{equation*}
    \begin{aligned}
        & w_1(\bu^\varepsilon)_t + \nabla_{\bu} w_1^\intercal \nabla_\bu \bF(\bu^\varepsilon) \bu^\varepsilon_x = \varepsilon\nabla_{\bu} w_1^\intercal \bu^\varepsilon_{xx},
    \end{aligned}
\end{equation*}
so using that $\nabla_\bu \bF$ is symmetric and $\nabla_\bu w_1 = \sqrt{2}\br_2$, there holds $\nabla_{\bu} w_1^\intercal \nabla_\bu \bF(\bu^\varepsilon) = \lambda_2 \nabla_\bu w_1^\intercal$, whence 
\begin{equation*}
    \begin{aligned}
        & w_1(\bu^\varepsilon)_t + \lambda_2(\bu^\varepsilon) w_1(\bu^\varepsilon)_x = \varepsilon w_1(\bu^\varepsilon)_{xx} - \varepsilon(\bu^\varepsilon_x)^\intercal \nabla_\bu^2 w_1 (\bu^\varepsilon) \bu^\varepsilon_x; 
    \end{aligned}
\end{equation*}
similarly, we get 
\begin{equation*}
    \begin{aligned}
        & w_2(\bu^\varepsilon)_t + \lambda_1(\bu^\varepsilon) w_2(\bu^\varepsilon)_x = \varepsilon w_2(\bu^\varepsilon)_{xx} - \varepsilon(\bu^\varepsilon_x)^\intercal \nabla_\bu^2 w_2 (\bu^\varepsilon) \bu^\varepsilon_x. 
    \end{aligned}
\end{equation*}
Note that since $w_j$ ($j=1,2)$ are affine, their Hessians are null, so the previous equations reduce to 
\begin{equation}\label{eq:RI heat eqns}
   \left\lbrace \begin{aligned}
        & w_1(\bu^\varepsilon)_t + (\log w_1(\bu^\varepsilon))_x = \varepsilon w_1(\bu^\varepsilon)_{xx}, \\ 
        & w_2(\bu^\varepsilon)_t - (\log w_2(\bu^\varepsilon))_x = \varepsilon w_2(\bu^\varepsilon)_{xx}, 
    \end{aligned}\right. 
\end{equation}
where we used that $$\lambda_2 = \frac{1}{w_1} \quad \text{and} \quad \lambda_1 = -\frac{1}{w_2}.$$ Each of the equations in \eqref{eq:RI heat eqns} is an autonomous strongly parabolic semilinear equation. An application of the Maximum Principle thereby ensures $$ c_0 \leq \essinf_\mathbb{R} w_j(\bu_0) \leq w_j(\bu^\varepsilon) \leq \esssup_\mathbb{R} w_j(\bu_0) \quad \text{a.e.} \quad (j=1,2).$$

The arguments presented above can only be made rigorous \emph{a posteriori}, \textit{i.e.}~one can only perform these manipulations once the existence of $\bu^\varepsilon$ is established. For this reason, we shall study the \emph{modified regularised problem}: 
\begin{equation}\label{eq:modified regularised problem}
   \left\lbrace \begin{aligned} &\bu^\varepsilon_t + \bM(\bu^\varepsilon) \bu^\varepsilon_x = \varepsilon \bu^\varepsilon_{xx}, \\ 
   &\bu^\varepsilon|_{t=0} = \bu_0, 
   \end{aligned}\right. 
\end{equation}
where, analogous to the diagonal representation \eqref{eq:diagonalised jacobian F}, $\bM$ is defined by 
\begin{equation}\label{eq:modified flux}
    \bM(\bu) \defeq -\frac{1}{2}\left( \begin{matrix}
        -1 & 1 \\ 
        1 & 1
    \end{matrix} \right) \left( \begin{matrix}
        \phi_1(\bu) & 0 \\ 
        0 & \phi_2(\bu) 
    \end{matrix} \right) \left( \begin{matrix}
        1 & -1 \\ 
        -1 & -1
    \end{matrix} \right), 
\end{equation}
where $\phi_1$ and $\phi_2$ are $C^1$ approximations of $\lambda_1$ and $\lambda_2$, respectively, with cut-off at the value $c_0$. In detail, we define, for all $\delta>0$, the function $\psi_\delta \in C^1(\mathbb{R})$ by 
\begin{equation}
    \label{psi_delta_definition}
    \begin{split}
    \psi_\delta(s)  \defeq  \left\lbrace \begin{aligned}
        & \frac{\delta}{2} \quad && \text{for } s \leq 0, \\ 
        & \frac{1}{2\delta}(s^2+\delta^2) \quad && \text{for } 0 < s \leq \delta, \\ 
        & s \quad && \text{for } s > \delta, 
    \end{aligned} \right.
    \end{split}
\end{equation}
and we note that $\delta/2 \leq \psi_\delta(s) \leq C(1+s)$, whence $\frac{1}{\psi_\delta} \in L^1_{\loc}(\mathbb{R})$. Define 
\begin{equation*}
    \phi_1(\bu)  \defeq  -\frac{1}{\psi_{c_0}(\sigma-\beta)}  \quad \text{and} \quad 
    \phi_2(\bu)  \defeq  \frac{1}{\psi_{c_0}(\sigma+\beta)}, 
\end{equation*}
and observe that there holds 
\begin{equation}\label{eq:modified flux coincides in good region}
    \bM(\bu) = \nabla_\bu \bF(\bu) \quad \text{whenever } w_1(\bu), \, w_2(\bu) \geq c_0. 
\end{equation}
Furthermore, it is clear from the diagonal representations \eqref{eq:diagonalised jacobian F} and \eqref{eq:modified flux} that $\bM$ has eigenvectors given by $\br_1, \, \br_2$, with eigenvalues given by $\phi_1, \, \phi_2$, respectively. In turn, the Riemann invariant coordinates $w_1, \, w_2$ from \eqref{eq:riemann invariants} are the same for the equation 
\begin{equation*}
    \bu_t + \bM(\bu) \bu_x = 0 
\end{equation*}
as for the original equation \eqref{eq:consvn law carroll gamma 3}. It follows that, in Riemann invariant coordinates, 
\begin{equation}\label{eq:eigenvalues in terms of riemann invariants}
    \phi_1 = -\frac{1}{\psi_{c_0}(w_2)}  , \quad 
    \phi_2 = \frac{1}{\psi_{c_0}(w_1)}. 
\end{equation}

Finally, observe that there holds 
\begin{equation*}
    0 < |\phi_j(\bu)| \leq \frac{2}{c_0} \quad (j=1,2), 
\end{equation*}
and note that $\bu \mapsto \bM(\bu)$ is $C^1$. Using the Cauchy--Schwarz inequality for the matrix operator norm, we have 
\begin{equation}\label{eq:op norm of auxiliary M}
    \vertiii{\bM(\bu)} \leq C \big( |\phi_1(\bu)| + |\phi_2(\bu)| \big) \leq C 
\end{equation}
for some positive constant $C$ depending only on $c_0$, and in particular independent of $\bu$.

\subsection{Existence for the Modified Regularised Problem}

We therefore turn to establishing the following result.

\begin{lemma}\label{lem:first existence using fixed point}
    Let $\varepsilon>0$ and $\bu_0 \in L^2(\mathbb{R})$. There exists $\bu^\varepsilon \in L^2_{\loc}(0,T;H^1(\mathbb{R}))$ with $\bu^\varepsilon_t \in L^2_{\loc}(0,T;H^{-1}(\mathbb{R}))$ be a weak solution of the modfiied regularised problem \eqref{eq:modified regularised problem}, \textit{i.e.} 
    \begin{equation*}
    \left\lbrace\begin{aligned}
        &\bu^\varepsilon_t + \bM(\bu^\varepsilon)\bu^\varepsilon_x = \varepsilon  \bu^\varepsilon_{xx}, \\ 
        &\bu^\varepsilon|_{t=0} = \bu_0. 
    \end{aligned}\right.
\end{equation*}
\end{lemma}

Our approach is classical; nevertheless, for clarity and completeness we give a self-contained exposition. We shall build the solution $\bu^\varepsilon$ using a fixed-point argument, by interpreting \eqref{eq:modified regularised problem} as a perturbation of the heat equation and using the Duhamel Principle. 

\smallskip 

Suppose that $\bu^\varepsilon$ is a solution of \eqref{eq:modified regularised problem}. Then $\bv(t,x) \defeq \bu^\varepsilon(t/\varepsilon,x)$ is a solution of 
\begin{equation*}
    \bv_t - \bv_{xx} = -\varepsilon^{-1}\bM(\bv)\bv_x 
\end{equation*}
with the same initial data $\bv_0 = \bu_0$. By the Duhamel Principle, $\bv$ must then verify the representation formula 
\begin{equation*}
    \bv(t,x) = \int_\mathbb{R} \Gamma(t,x-y) \bu_0(y) \d y - \varepsilon^{-1}\int_0^t \int_\mathbb{R} \Gamma(t-s,x-y) \bM(\bv(s,y)) \bv_y(s,y) \d y \d s, 
\end{equation*}
where $\Gamma$ is the one-dimensional heat kernel, \textit{i.e.} 
\begin{equation*}
    \Gamma(t,x) = \frac{1}{\sqrt{4\pi t}}e^{-\frac{x^2}{4t}}. 
\end{equation*}
Rewriting in terms of $\bu$, this formula corresponds to 
\begin{equation*}
    \bu^\varepsilon(t,x) = \int_\mathbb{R} \Gamma(\varepsilon t,x-y) \bu_0(y) \d y - \int_0^{t} \int_\mathbb{R} \Gamma(\varepsilon (t - s),x-y) \bM(\bu^\varepsilon(s,y))\bu^\varepsilon_y(s,y) \d y \d s. 
\end{equation*}
Define the mapping 
\begin{equation}\label{eq:L def}
    L\bu(t,x) \defeq \int_\mathbb{R} \Gamma(\varepsilon t,x-y) \bu_0(y) \d y - \int_0^{t} \int_\mathbb{R} \Gamma(\varepsilon(t-s),x-y) \bM(\bu(s,y)) \bu_y(s,y) \d y \d s. 
\end{equation}
Our objective is to show that there exists a fixed point $\bu_* = L\bu_*$; such a $\bu_*$ is manifestly a solution of \eqref{eq:modified regularised problem}, as direct computation shows that $\bw = L\bu$ solves 
\begin{equation*}
    \left\lbrace\begin{aligned}
    &\bw_t - \varepsilon\bw_{xx} = - \bM(\bu)\bu_x, \\ 
    &\bw|_{t=0} = \bu_0. 
    \end{aligned}\right. 
\end{equation*}

\smallskip

In what follows we will employ Schaefer's Fixed Point Theorem: 
\begin{theorem}[Schaefer's Theorem]\label{thm:schaefer}
    Let $X$ be a Banach space and $L:X \to X$ a continuous and compact mapping of $X$ into itself, such that the set 
    \begin{equation*}
      A  \defeq   \Big\{ \bv \in X : \, \bw = \lambda L\bw \text{ for some } \lambda \in [0,1] \Big\} 
    \end{equation*}
    is bounded in $X$. Then $L$ admits a fixed point $\bw_* \in X$ such that $L\bw_* = \bw_*$. 
\end{theorem}

Let $T>0$, and define the Banach space 
\begin{equation*}
    X \defeq \Big\{ \bv \in L^2(0,T;H^1(\mathbb{R})) : \, \bv_t \in L^2(0,T;H^{-1}(\mathbb{R})) \Big\}, 
\end{equation*}
and, for $\lambda \in [0,1]$, 
\begin{equation*}
    A_\lambda \defeq \big\{ \bw \in X : \, \bw = \lambda L\bw \big\}. 
\end{equation*}
Note that 
\begin{equation}\label{eq:identification Alambda}
        \Big\{ \bw \in X : \, \bw = \lambda L\bw \text{ for some } \lambda \in [0,1] \Big\} = \bigcup_{\lambda \in [0,1]} A_\lambda. 
    \end{equation}

\begin{lemma}\label{lem:fixed pt i}
    The mapping $L : X \to X$ is continuous and compact.
\end{lemma}

\begin{proof}
We divide the proof into several steps. 

\smallskip

 \noindent 1. \textit{$L$ is self-mapping, i.e.~$L(X) \subset X$.}  Let $\bw \defeq L\bv$, and note that there holds 
\begin{equation*}
    \left\lbrace\begin{aligned}
    &\bw_t - \varepsilon\bw_{xx} = - \bM(\bv)\bv_x, \\ 
    &\bw|_{t=0} = \bu_0 
    \end{aligned}\right. 
\end{equation*}
in the weak sense. In turn, performing the standard parabolic estimate by testing the equation with $\bw$, we obtain 
\begin{equation*}
    \frac{1}{2}\frac{\der}{\der t}\int_\mathbb{R} |\bw|^2 \d x + \varepsilon \int_\mathbb{R} |\bw_x|^2 \d x = -\int_{\mathbb{R}} \bw \cdot \bM(\bv)\bv_x \d x \leq \frac{1}{2}\int_\mathbb{R} |\bw|^2 \d x + C\int_\mathbb{R}|\bv_x|^2 \d x, 
\end{equation*}
and Gr\"onwall's Lemma implies, for a.e.~$t \in (0,T)$, 
\begin{equation*}
    \Vert \bw (t,\cdot) \Vert^2_{L^2(\mathbb{R})} \leq \Vert \bu_0 \Vert^2_{L^2(\mathbb{R})} \exp\Big( C \int_0^t \Vert \bv_x(s,\cdot) \Vert^2_{L^2(\mathbb{R})} \d s \Big). 
\end{equation*}
It therefore follows that 
\begin{equation}\label{eq:L2H1 est fixed point}
    \begin{aligned} \Vert \bw \Vert^2_{L^\infty(0,T;L^2(\mathbb{R}))} + & \varepsilon \Vert \bw_x \Vert^2_{L^2((0,T)\times\mathbb{R})} \\ 
    &\leq C(1+T)\Vert \bu_0 \Vert^2_{L^2(\mathbb{R})} \exp\Big( C \Vert \bv_x \Vert^2_{L^2((0,T)\times\mathbb{R})}  \Big) + CT \Vert \bv_x \Vert^2_{L^2((0,T)\times\mathbb{R})}, 
    \end{aligned}
\end{equation}
whence $L$ is self-mapping since $L^\infty(0,T;L^2(\mathbb{R})) \subset L^2(0,T;L^2(\mathbb{R}))$ for $T$ finite.

\smallskip 
\noindent 2. \textit{Estimate on time derivatives}. We begin with an estimate on the time derivative which will enable us to employ the Aubin--Lions Lemma. Let $\bzeta \in L^2(0,T;H^1(\mathbb{R}))$. Testing against the equation, we obtain 
\begin{equation*}
   \begin{aligned}
       |\langle \bw_t , \bzeta \rangle| &\leq \varepsilon \Big| \int_0^T \int_\mathbb{R} \bw_x \cdot \bzeta_x \d x \d t \Big| + \Big| \int_0^T \int_\mathbb{R} \bzeta \cdot \bM(\bv) \bv_x \d x \d t \Big| \\ 
       &\leq \varepsilon \Vert \bw_x \Vert_{L^2((0,T)\times\mathbb{R})} \Vert \bzeta_x \Vert_{L^2((0,T)\times\mathbb{R})} + C\Vert \bzeta \Vert_{L^2((0,T)\times\mathbb{R})} \Vert \bv_x \Vert_{L^2((0,T)\times\mathbb{R})}, 
   \end{aligned} 
\end{equation*}
whence 
\begin{equation*}
   \begin{aligned}
       \Vert \bw_t \Vert_{L^2(0,T;H^{-1}(\mathbb{R}))} \leq \varepsilon \Vert \bw_x \Vert_{L^2((0,T)\times\mathbb{R})} + C\Vert \bv_x \Vert_{L^2((0,T)\times\mathbb{R})}, 
   \end{aligned} 
\end{equation*}
\textit{i.e.}, using \eqref{eq:L2H1 est fixed point}, 
\begin{equation}\label{eq:time deriv est fixed pt}
   \begin{aligned}
       \Vert \bw_t \Vert_{L^2(0,T;H^{-1}(\mathbb{R}))} \leq &\sqrt{\varepsilon}C(1+T)^{\frac{1}{2}}\Vert \bu_0 \Vert_{L^2(\mathbb{R})} \exp\Big( C \Vert \bv_x \Vert^2_{L^2((0,T)\times\mathbb{R})} \big) \\ 
       &+ C(1+T^{\frac{1}{2}}) \Vert \bv_x \Vert_{L^2((0,T)\times\mathbb{R})}. 
   \end{aligned} 
\end{equation}

\smallskip 
\noindent 3. \textit{Weak compactness in $X$}. Let $\{\bv_n\}_n$ be a bounded sequence in $X$ and define $\bw_n  \defeq  L\bv_n$ for all $n$. The Aubin--Lions Lemma implies that, up to a subsequence which we do not relabel, there holds
\begin{equation}
    \bv_n \to \bv \quad \text{weakly in } X, \text{ strongly in } L^2((0,T)\times\mathbb{R}), \text{ and a.e.} 
\end{equation}
for some $\bv \in X$, where we also used Alaoglu's Theorem and the fact that $X$ is reflexive. Furthermore, the estimates \eqref{eq:L2H1 est fixed point} and \eqref{eq:time deriv est fixed pt} imply that $\{\bw_n\}_n$ also satisfies the assumptions of the Aubin--Lions Lemma, so there exists $\bw \in X$ such that, up to a subsequence which we do not relabel, using also Alaoglu's Theorem, 
\begin{equation}\label{eq:conv from aubin lions}
    \bw_n \to \bw \quad \text{weakly in } X, \text{ strongly in } L^2((0,T)\times\mathbb{R}), \text{ and a.e.} 
\end{equation}

\smallskip 
\noindent 4. \textit{Strong compactness in $L^2(0,T;H^1(\mathbb{R}))$}. It remains to show that the convergence also occurs strongly with respect to the topology of $X$, \textit{i.e.}~it remains to show that $\lim_n \Vert (\bw_n)_x - \bw_x \Vert_{L^2((0,T)\times\mathbb{R})} = 0$ and an analogous strong convergence for the time derivatives in the space $L^2(0,T;H^{-1}(\mathbb{R}))$. Concerning the first convergence, since we already know that $\{\bw_n\}_n$ converges weakly to $\bw$ in $L^2(0,T;H^1(\mathbb{R}))$, it suffices to show 
\begin{equation}\label{eq:required strong conv for cpct fixed pt}
    \lim_n \Vert (\bw_n)_x \Vert_{L^2((0,T)\times\mathbb{R})} = \Vert \bw_x \Vert_{L^2((0,T)\times\mathbb{R})}. 
\end{equation}
The weak formulation for $\bw_n$ reads 
\begin{equation}\label{eq:weak form wn fixed pt}
   -\int_0^T \int_\mathbb{R} \bw_n \cdot \bzeta_t \d x \d t + \varepsilon \int_0^T \int_\mathbb{R} (\bw_n)_x \cdot \bzeta_x \d x \d t = - \int_0^T \int_\mathbb{R} \bM(\bv_n) (\bv_n)_x \cdot \bzeta \d x \d t 
\end{equation}
for all $\bzeta \in C^\infty([0,T]\times\mathbb{R})$ such that $\bzeta(0,\cdot)=\bzeta(T,\cdot) = 0$. The matrix $\bM$ is continuous, so $\bv_n \to \bv$ a.e.~implies that $\bM(\bv_n) \to \bM(\bv)$ a.e. Furthermore, the bound \eqref{eq:op norm of auxiliary M} gives $\vertiii{\bM(\bv_n)} \leq C$, whence $C|\bzeta| \in L^1((0,T)\times\mathbb{R})$ is a suitable dominating function for $\bzeta^\intercal \bM(\bv_n)$ which converges a.e.~to $\bzeta^\intercal \bM(\bv)$. The Dominated Convergence Theorem then yields 
\begin{equation*}
    \bzeta^\intercal \bM(\bv_n) \to \bzeta^\intercal \bM(\bv) \quad \text{strongly in } L^p((0,T)\times\mathbb{R}) \text{ for all } p \in [1,\infty). 
\end{equation*}
Since $(\bv_n)_x \to \bv_x$ weakly in $L^2((0,T)\times\mathbb{R})$, it follows that there holds the convergence for the product 
\begin{equation*}
    \bM(\bv_n) (\bv_n)_x \cdot \bzeta \rightharpoonup \bM(\bv) \bv_x \cdot \bzeta \quad \text{weakly in } L^1((0,T)\times\mathbb{R}). 
\end{equation*}
Using the above and the aforementioned weak convergences, we deduce from \eqref{eq:weak form wn fixed pt} that there holds 
\begin{equation}\label{eq:weak form w fixed pt}
   -\int_0^T \int_\mathbb{R} \bw \cdot \bzeta_t \d x \d t + \varepsilon \int_0^T \int_\mathbb{R} \bw_x \cdot \bzeta_x \d x \d t = - \int_0^T \int_\mathbb{R} \bM(\bv) \bv_x \cdot \bzeta \d x \d t. 
\end{equation}
A standard approximation argument allows us to introduce $\bw_n$ into the weak formulation \eqref{eq:weak form wn fixed pt} and $\bw$ into \eqref{eq:weak form w fixed pt} on an arbitrary time interval $[t_1,t_2] \subsetneq [0,T]$. Subtracting the two resulting equations, we get 
\begin{equation}\label{eq:diff of weak forms fixed pt}
   \begin{aligned} \frac{1}{2}&\int_\mathbb{R} \big( |\bw_n(t_2,x)|^2 -  |\bw(t_2,x)|^2 \big) \d x - \frac{1}{2}\int_\mathbb{R} \big(|\bw_n(t_1,x)|^2 - |\bw(t_1,x)|^2 \big) \d x \\ 
   &= \varepsilon \int_{t_1}^{t_2} \int_\mathbb{R} \big( |(\bw_n)_x|^2 - |\bw_x|^2 \big) \d x \d t -\int_{t_1}^{t_2}\int_\mathbb{R} \Big( \bM(\bv_n) (\bv_n)_x \cdot \bv_n - \bM(\bv) \bv_x \! \cdot \! \bv \Big) \d x \d t. 
   \end{aligned}
\end{equation}
Recall that the convergence $\bw_n \to \bw$ strongly in $L^2((0,T)\times\mathbb{R})$ implies, up to a subsequence which we do not relabel, 
\begin{equation*}
    \bw_n(t,\cdot) \to \bw(t,\cdot) \quad \text{strongly in } L^2(\mathbb{R}) \text{ for a.e.~}t. 
\end{equation*}
It follows that the left-hand side of \eqref{eq:diff of weak forms fixed pt} vanishes in the limit as $n\to\infty$. Hence, using also the continuity of $s \mapsto |s|$ to interchange the limit with the absolute value, 
\begin{equation}\label{eq:diff of weak forms fixed pt checkpoint}
      \begin{aligned} \lim_n \Big| \Vert (\bw_n )_x \Vert^2_{L^2((t_1,t_2)\times\mathbb{R})} -&  \Vert \bw_x \Vert^2_{L^2((t_1,t_2)\times\mathbb{R})}\Big| \\ 
      &= \varepsilon^{-1} \bigg|\lim_n \int_{t_1}^{t_2}\int_\mathbb{R} \Big( \bM(\bv_n) (\bv_n)_x \cdot \bv_n - \bM(\bv) \bv_x \cdot \bv \Big) \d x \d t\bigg|. 
   \end{aligned}
\end{equation}
For the term on the right-hand side, we write
\begin{equation}\label{eq:using dct to conclude fixed pt cpctness}
    \begin{aligned}
      \Vert \bv_n^\intercal \bM(\bv_n) - \bv^\intercal \bM(\bv) \Vert_{L^2} &\leq \Vert \bM(\bv_n) \cdot (\bv_n - \bv) \Vert_{L^2} + \Vert \big( \bM(\bv_n) - \bM(\bv) \big) \cdot \bv \Vert_{L^2} \\ 
      &\leq C \Vert \bv_n - \bv \Vert_{L^2} + \bigg( \int_{0}^{T}\int_\mathbb{R} \vertiii{\bM(\bv_n) - \bM(\bv)}^2 |\bv|^2 \d x \d t \bigg)^{\frac{1}{2}}. 
    \end{aligned}
\end{equation}
The first term on the right-hand side manifestly vanishes due to the strong convergence in $L^2$ for the sequence $\{\bv_n\}_n$. Recall that $\bM(\bv_n) \to \bM(\bv)$ a.e.~ and $\vertiii{\bM(\bv_n) - \bM(\bv)}^2 \leq C$, so the Dominated Convergence Theorem implies that the second term on the right-hand side vanishes in the limit as $n\to\infty$. It follows that 
\begin{equation*}
    \begin{aligned}
      \bv_n^\intercal \bM(\bv_n) \to \bv^\intercal \bM(\bv) \quad \text{strongly in } L^2((0,T)\times\mathbb{R}). 
    \end{aligned}
\end{equation*}
Since $(\bv_n)_x \rightharpoonup \bv_x$ weakly in $L^2((0,T)\times\mathbb{R})$, it follows that 
\begin{equation*}
    \bM(\bv_n) (\bv_n)_x \cdot \bv_n \rightharpoonup \bM(\bv) \bv_x \cdot \bv \quad \text{weakly in } L^1((0,T)\times\mathbb{R}), 
\end{equation*}
and hence 
\begin{equation*}
      \begin{aligned} 
      \lim_n \int_{t_1}^{t_2}\int_\mathbb{R} \Big( \bM(\bv_n) (\bv_n)_x \cdot \bv_n - \bM(\bv) \bv_x \cdot \bv \Big) \d x \d t = 0. 
   \end{aligned}
\end{equation*}
Returning to \eqref{eq:diff of weak forms fixed pt checkpoint}, we obtain 
\begin{equation*}
    \lim_n \Vert (\bw_n)_x \Vert_{L^2((t_1,t_2)\times\mathbb{R})} = \Vert \bw_x \Vert^2_{L^2((t_1,t_2)\times\mathbb{R})} \quad \text{for a.e.~}t_1, \, t_2. 
\end{equation*}
We now relax the assumption $[t_1,t_2] \subsetneq [0,T]$ (by taking a further subsequence if necessary), and deduce 
\begin{equation*}
    \lim_n \Vert (\bw_n)_x \Vert_{L^2((0,T)\times\mathbb{R})} = \Vert \bw_x \Vert^2_{L^2((0,T)\times\mathbb{R})}; 
\end{equation*}
this, in addition to the weak convergence in $L^2(0,T;H^1(\mathbb{R}))$, implies 
\begin{equation}\label{eq:strong conv L2H1 for fixed pt}
    \bw_n \to \bw \quad \text{strongly in } L^2(0,T;H^1(\mathbb{R})). 
\end{equation}

\smallskip 
\noindent 5. \textit{Strong compactness of time derivatives in $L^2(0,T;H^{-1}(\mathbb{R}))$}. It remains to show that $(\bw_n)_t \to \bw_t$ strongly in $L^2(0,T;H^{-1}(\mathbb{R}))$. Taking the difference of the two weak formulations, there holds 
\begin{equation*}
    \begin{aligned}
        |\langle (\bw_n)_t \! & - \! \bw_t , \bzeta \rangle| \\ 
        &\leq \varepsilon \int_0^T \int_\mathbb{R} \big| \big( (\bw_n)_x \! -\! \bw_x \big) \cdot \bzeta_x \big| \d x \d t + \int_0^T \int_\mathbb{R} \big| \big(\bM(\bv_n)(\bv_n)_x \! -\! \bM(\bv)\bv_x \big) \! \cdot \! \bzeta \big| \d x \d t \\ 
        &\leq \Big(\varepsilon \Vert \bw_n \! - \! \bw \Vert_{L^2(0,T;H^1(\mathbb{R}))} \!+\! \Vert \bM(\bv_n)(\bv_n)_x \!-\! \bM(\bv)\bv_x
 \Vert_{L^2((0,T)\times\mathbb{R})}\Big)\Vert \bzeta \Vert_{L^2(0,T;H^1(\mathbb{R}))}
    \end{aligned}
\end{equation*}
for all $\bzeta \in L^2(0,T;H^1(\mathbb{R}))$. We compute 
\begin{equation*}
    \begin{aligned}
        \Vert \bM(\bv_n)(\bv_n)_x - \bM(\bv)&\bv_x
 \Vert_{L^2((0,T)\times\mathbb{R})} \\ &\leq \Vert \bM(\bv_n)(\bv_n - \bv)_x \Vert_{L^2((0,T)\times\mathbb{R})} + \Vert (\bM(\bv) - \bM(\bv_n))\bv_x
 \Vert_{L^2((0,T)\times\mathbb{R})} \\ 
 &\leq C \Vert \bv_n \! - \! \bv \Vert_{L^2(0,T;H^1(\mathbb{R}))} + \Vert (\bM(\bv_n)\! -\! \bM(\bv) ) \bv_x\Vert_{L^2((0,T)\times\mathbb{R})}
    \end{aligned}
\end{equation*}
so that there holds 
\begin{equation*}
    \begin{aligned}
        \Vert (\bw_n)_t - \bw_t \Vert_{L^2(0,T;H^{-1}(\mathbb{R}))} \leq \, & \varepsilon \Vert \bw_n - \bw \Vert_{L^2(0,T;H^1(\mathbb{R}))} + C\Vert \bv_n - \bv \Vert_{L^2(0,T;H^1(\mathbb{R}))} \\ 
        &+ \Vert (\bM(\bv_n)\! -\! \bM(\bv) ) \bv_x\Vert_{L^2((0,T)\times\mathbb{R})}, 
    \end{aligned}
\end{equation*}
giving
\begin{equation*}
    \limsup_n  \Vert (\bw_n)_t - \bw_t \Vert_{L^2(0,T;H^{-1}(\mathbb{R}))} = \limsup_n \Vert (\bM(\bv_n)\! -\! \bM(\bv) ) \bv_x\Vert_{L^2((0,T)\times\mathbb{R})}. 
\end{equation*}
Arguing as we did to control the second term on the right-hand side of the final line of \eqref{eq:using dct to conclude fixed pt cpctness}, we use the Dominated Convergence Theorem to deduce
\begin{align*}
    \limsup_n \Vert (\bM(\bv_n)\! -\! \bM(\bv) ) \bv_x\Vert_{L^2((0,T)\times\mathbb{R})} 
    &\leq \limsup_n \bigg(\int_0^T \int_\mathbb{R} \vertiii{\bM(\bv_n) - \bM(\bv)}^2 |\bv_x|^2 \d x \d t \bigg)^{\frac{1}{2}} \\ 
    &= 0, 
\end{align*}
and it follows that 
\begin{equation}\label{eq:time deriv cpctness fixed pt}
    (\bw_n)_t \to \bw_t  \quad \text{strongly in } L^2(0,T;H^{-1}(\mathbb{R})). 
\end{equation}

\smallskip 
\noindent 6. \textit{Compactness in $X$ and continuity}. By combining \eqref{eq:strong conv L2H1 for fixed pt} and \eqref{eq:time deriv cpctness fixed pt}, we have shown that, given any bounded sequence $\{\bv_n\}_n$ in $X$, the sequence $\{L\bv_n \}_n$ admits a strongly convergent subsequence in $X$, \textit{i.e.}~the mapping $L:X \to X$ is compact. The continuity of $L$ is immediate from the compactness and the uniqueness of limits, \textit{i.e.} given a sequence $\{\bv_n \}_n$ in $X$ converging to $\bv \in X$, the previous proof shows that $\{\bw_n=L\bv_n\}_n$ converges to $\bw=L\bv \in X$. 
\end{proof}

\begin{lemma}\label{lem:fixed pt ii}
    The set $\big\{ \bv \in X : \, \bv = \lambda L\bv \text{ for some } \lambda \in [0,1] \big\}$ is bounded in $X$. 
\end{lemma}

\begin{proof}
\noindent 1. \textit{Case $\lambda \neq 0$}. Let $\bv \in A_\lambda$ for $\lambda > 0$. Then $\bv$ solves the equation 
\begin{equation*}
    \left\lbrace  \begin{aligned}
        & \bv_t - \varepsilon\bv_{xx} =  -\lambda \bM(\bv)\bv_x, \\ 
        & \bv|_{t=0} = \lambda \bu_0. 
    \end{aligned} \right. 
\end{equation*}
It follows that the usual parabolic estimate, performed by testing the equation against $\bv$ and using the boundedness of $\bM$ from \eqref{eq:op norm of auxiliary M}, yields 
\begin{equation*}
 \begin{aligned}   \frac{1}{2}\frac{\der}{\der t}\int_{\mathbb{R}} |\bv|^2 \d x + \varepsilon \int_{\mathbb{R}} |\bv_x|^2 \d x &\leq C\lambda\int_{\mathbb{R}} |\bv| |\bv_x | \d x \leq \frac{\varepsilon}{2}\int_\mathbb{R} |\bv_x|^2 \d x + C\lambda^2 \int_\mathbb{R} |\bv|^2 \d x ,
 \end{aligned}
\end{equation*}
where we used Young's inequality. We deduce that for all $t \in [0,T]$
\begin{equation}\label{eq:pre gronwall fixed point for Alambda bounded}
   \begin{aligned} \int_{\mathbb{R}} |\bv|^2 \d x + \varepsilon \int_0^t \int_{\mathbb{R}} |\bv_x|^2 \d x \d t' &\leq \lambda^2 \int_{\mathbb{R}} |\bu_0|^2 \d x + C\lambda^2 \int_0^t \int_{\mathbb{R}} |\bv|^2 \d x \d t', 
   \end{aligned}
\end{equation}
whence an application of Gr\"onwall's Lemma gives 
\begin{equation*}
   \begin{aligned} \Vert \bv(t,\cdot) \Vert^2_{L^2(\mathbb{R})} \leq \lambda^2 \Vert \bu_0 \Vert^2_{L^2(\mathbb{R})} e^{C\lambda^2 t} \quad \text{for a.e.~} t \in [0,T]. 
   \end{aligned}
\end{equation*}
By returning to \eqref{eq:pre gronwall fixed point for Alambda bounded}, we deduce that for all $\lambda > 0$ there holds 
\begin{equation}\label{eq:for lambda nonzero L2H1 bound}
    \Vert \bv \Vert_{L^2(0,T;H^1(\mathbb{R}))}^2 \leq C\lambda^2 \Vert \bu_0 \Vert_{L^2(\mathbb{R})}^2 e^{C \lambda^2 T} \quad \text{for } \bv \in A_\lambda, 
\end{equation}
for some positive constant $C$ independent of $\lambda, \, \varepsilon, \, T$. 

Returning to the equation, there holds, for all $\bzeta \in L^2(0,T;H^1(\mathbb{R}))$, 
\begin{equation*}
   \begin{aligned} |\langle \bv_t , \bzeta \rangle | &\leq \varepsilon \int_0^T \int_\mathbb{R} |\bv_x \cdot \bzeta_x| \d x \d t + \lambda \int_0^T \int_\mathbb{R} |\bM(\bv) \bv_x \cdot \bzeta| \d x \d t \\ 
   &\leq \varepsilon \Vert \bv_x \Vert_{L^2((0,T)\times\mathbb{R})}\Vert \bzeta_x \Vert_{L^2((0,T)\times\mathbb{R})} + C\lambda \Vert \bv_x \Vert_{L^2((0,T)\times\mathbb{R})} \Vert \bzeta \Vert_{L^2((0,T)\times\mathbb{R})}, 
   \end{aligned}
\end{equation*}
and thus 
\begin{equation*}
   \begin{aligned} \Vert \bv_t \Vert_{L^2(0,T;H^{-1}(\mathbb{R}))} \leq (\varepsilon + C\lambda) \Vert \bv_x \Vert_{L^2(0,T;H^1(\mathbb{R}))}, 
   \end{aligned}
\end{equation*}
which is bounded in view of \eqref{eq:for lambda nonzero L2H1 bound}. By combining the previous bound with \eqref{eq:for lambda nonzero L2H1 bound}, we deduce 
\begin{equation}\label{eq:for lambda nonzero X bound}
    \Vert \bv \Vert_{X} \leq C\lambda(\varepsilon+\lambda) \Vert \bu_0 \Vert_{L^2(\mathbb{R})} e^{C \lambda^2 T} \quad \text{for } \bv \in A_\lambda, 
\end{equation}
for some positive constant $C$ independent of $\lambda, \, \varepsilon, \, T$. 

\smallskip 

\noindent 2. \textit{Case $\lambda = 0$}. In this case, $\bv$ solves the heat equation $\bv_t - \bv_{xx}=0$ with zero initial data. Uniqueness implies that $\bv \equiv 0$, whence 
\begin{equation}\label{eq:for lambda zero X bound}
    \Vert \bv \Vert_X = 0 \quad \text{for } \bv \in A_0. 
\end{equation}

\smallskip 

\noindent 3. \textit{Conclusion}. From \eqref{eq:for lambda nonzero X bound} and \eqref{eq:for lambda zero X bound}, we deduce that there holds 
\begin{equation*}
    \Vert \bv \Vert_{X} \leq C\lambda(\varepsilon+\lambda) \Vert \bu_0 \Vert_{L^2(\mathbb{R})} e^{C \lambda^2 T} \quad \text{for } \bv \in A_\lambda, 
\end{equation*}
for all $\lambda \in [0,1]$, and thus there exists a positive constant $C_{\varepsilon,T}$ independent of $\lambda$ such that 
\begin{equation*}
    \Vert \bv \Vert_{X} \leq C_{\varepsilon,T} \quad \text{for } \bv \in A_\lambda 
\end{equation*}
for all $\lambda \in [0.1]$. Using the identification \eqref{eq:identification Alambda}, it follows that $ \big\{ \bv \in X : \, \bv = \lambda L\bv \text{ for some } \lambda \in [0,1] \big\} $ is a bounded set. 
\end{proof}

We are now ready to give the proof of Lemma \ref{lem:first existence using fixed point}. 

\begin{proof}[Proof of Lemma \ref{lem:first existence using fixed point}]
    Let $T>0$ be arbitrary. Lemmas \ref{lem:fixed pt i} and \ref{lem:fixed pt ii} imply that the conditions of Schaefer's Fixed Point Theorem (Theorem \ref{thm:schaefer}) are satisfied, so there exists $\bu^\varepsilon \in X$ a weak solution of the modified regularised problem \eqref{eq:modified regularised problem}, 
\begin{equation*}
    \left\lbrace \begin{aligned}
        &\bu^\varepsilon_t + \bM(\bu^\varepsilon) \bu^\varepsilon_x = \varepsilon \bu^\varepsilon_{xx}, \\ 
        &\bu^\varepsilon|_{t=0} = \bu_0 
    \end{aligned} \right. 
\end{equation*}
on the interval $[0,T]$. We repeat this procedure on the intervals $[nT,(n+1)T]$ using the initial data $\bu^\varepsilon(nT,\cdot)$ for $n \geq 1$ to construct a globally-defined $\bu^\varepsilon \in L^2_{\loc}(0,T;H^1(\mathbb{R}))$ weak solution. 
\end{proof}

\subsection{Invariant Regions}
This subsection is devoted to rigorously proving the formal invariant region argument presented at the start of \S \ref{sec:unif est}. We establish these invariant regions first for the modified regularised problem \eqref{eq:modified regularised problem}, and then prove Proposition \ref{prop:approx pbms}.

\begin{lemma}\label{lem:to prove the prop}
        Let $\varepsilon>0$ and $\bu_0 \in L^2(\mathbb{R}) \cap L^\infty(\mathbb{R})$ satisfy 
    \begin{equation*}
        \essinf_\mathbb{R} w_1(\bu_0) , ~ \essinf_\mathbb{R} w_2(\bu_0) \geq c_0 
    \end{equation*}
    and let $\bu^\varepsilon \in L^\infty(\mathbb{R}^2_+) \cap L^2_{\loc}(0,T;H^1(\mathbb{R}))$ with $\bu^\varepsilon_t \in L^2_{\loc}(0,T;H^{-1}(\mathbb{R}))$ be a weak solution of the modified regularised problem \eqref{eq:modified regularised problem}. Then $\bu^\varepsilon$ is smooth away from the initial time, is a classical solution, satisfies 
\begin{equation}\label{eq:RI lower bound lem to prove prop}
        w_1(\bu^\varepsilon) , ~ w_2(\bu^\varepsilon) \geq c_0 \quad \text{a.e.~in }\mathbb{R}^2_+, 
    \end{equation}
    and there exists a positive constant $C$ independent of $\varepsilon$ such that 
    \begin{equation*}
        \Vert \bu^\varepsilon \Vert_{L^\infty(\mathbb{R}^2_+)} \leq C. 
    \end{equation*}
\end{lemma}

\begin{proof}

The proof consists of several steps. 

\smallskip 
\noindent 1. \textit{Scalar equations for the Riemann invariants}. We argue as we did in order to obtain \eqref{eq:RI heat eqns}. By multiplying the equation for the modified regularised problem \eqref{eq:modified regularised problem} with $\nabla_{\bu}w_j^\intercal$ ($j=1,2$), we find that there holds 
\begin{equation*}
   \left\lbrace \begin{aligned}
        & w_1(\bu^\varepsilon)_t + \phi_2(\bu^\varepsilon) w_1(\bu^\varepsilon)_x = \varepsilon w_1(\bu^\varepsilon)_{xx}, \\ 
        & w_2(\bu^\varepsilon)_t + \phi_1(\bu^\varepsilon) w_2(\bu^\varepsilon)_x = \varepsilon w_2(\bu^\varepsilon)_{xx} 
    \end{aligned}\right. 
\end{equation*}
in the weak sense. Observe that, by defining the primitive 
\begin{equation*}
    h(s)  \defeq  \int_0^s \frac{1}{\psi_{c_0}(\tau)} \d \tau, 
\end{equation*}
where $\psi_{c_0}$ is as defined in \eqref{psi_delta_definition} and we note that $h \in C^2(\mathbb{R})$, using \eqref{eq:eigenvalues in terms of riemann invariants}, the equations for the Riemann invariants may be rewritten as 
\begin{equation}\label{eq:scalar autonomous}
 \left\lbrace  \begin{aligned}
       &w_1(\bu^\varepsilon)_t + h(w_1(\bu^\varepsilon))_x = \varepsilon w_1(\bu^\varepsilon)_{xx}, \\ 
       &w_2(\bu^\varepsilon)_t - h(w_2(\bu^\varepsilon))_x = \varepsilon w_2(\bu^\varepsilon)_{xx}. 
   \end{aligned} \right. 
\end{equation}

\smallskip 
\noindent 2. \textit{Application of the Maximum Principle}. Applying a standard bootstrapping argument to \eqref{eq:scalar autonomous} using also the De Giorgi--Nash--Moser Theorem shows that $w_j$ ($j=1,2$) are smooth away from the initial time; we omit the details for brevity. In turn, $w_1,w_2$ are classical solutions of \eqref{eq:scalar autonomous}. Furthermore, using the formulas $\sigma = \frac{1}{2}(w_1+w_2)$ and $\beta = \frac{1}{2}(w_1-w_2)$, it follows that $\bu^\varepsilon$ is smooth away from the initial time and is a classical solution of \eqref{eq:modified regularised problem}. 

It then follows from the fact that $w_j$ ($j=1,2$) are classical solutions of the autonomous scalar equations \eqref{eq:scalar autonomous} that we may apply the Maximum Principle (\textit{cf.~e.g.}~\cite[\S 2.3]{Serre1}). We deduce 
\begin{equation*}
\essinf_\mathbb{R} w_j(\bu_0) \leq w_j(\bu) \leq \esssup_\mathbb{R} w_j(\bu_0) \quad (j=1,2). 
\end{equation*}
Since the initial data is chosen such that $c_0 \leq \essinf_\mathbb{R} w_j(\bu_0)$, we obtain \eqref{eq:RI lower bound lem to prove prop}. Furthermore, it follows from $\sigma = \frac{1}{2}(w_1+w_2)$ and $\beta = \frac{1}{2}(w_1-w_2)$ that there holds a.e.~in $\mathbb{R}^2_+$ 
\begin{equation*}
   \begin{aligned} 
   &\frac{1}{2}\big(\essinf_\mathbb{R} w_1(\bu_0) + \essinf_\mathbb{R} w_2(\bu_0) \big) \leq \sigma^\varepsilon(t,x) \leq \frac{1}{2}\big(\esssup_\mathbb{R} w_1(\bu_0) + \esssup_\mathbb{R} w_2(\bu_0) \big), \\ 
   &\frac{1}{2}\big(\essinf_\mathbb{R} w_1(\bu_0) - \esssup_\mathbb{R} w_2(\bu_0) \big) \leq \beta^\varepsilon(t,x) \leq \frac{1}{2}\big(\esssup_\mathbb{R} w_1(\bu_0) - \essinf_\mathbb{R} w_2(\bu_0) \big), 
   \end{aligned}
\end{equation*}
and thus there exists a positive constant $C$, independent of $\varepsilon$, such that $\Vert \bu^\varepsilon \Vert_{L^\infty(\mathbb{R}^2_+)} \leq C$, which completes the proof. 
\end{proof}

The proof of Proposition \ref{prop:approx pbms} now follows as a simple corollary of Lemmas \ref{lem:first existence using fixed point} and \ref{lem:to prove the prop}.

\begin{proof}[Proof of Proposition \ref{prop:approx pbms}]
The existence of a global weak solution $\bu^\varepsilon$ follows from Lemma \ref{lem:first existence using fixed point}. Then, using Lemma \ref{lem:to prove the prop}, we get 
\begin{equation*}
c_0 \leq w_j(\bu) \quad \text{a.e.~in } \mathbb{R}^2_+ \quad (j=1,2), 
\end{equation*}
whence \eqref{eq:modified flux coincides in good region} implies, using the chain rule (which is justified due to the regularity of $h$ and $\bu^\varepsilon$), that
\begin{equation*}
    \bF(\bu^\varepsilon)_x = \nabla_\bu \bF(\bu^\varepsilon) \bu^\varepsilon_x = \bM(\bu^\varepsilon) \bu^\varepsilon_x \quad \text{a.e.~in } \mathbb{R}^2_+. 
\end{equation*}
Thus $\bu^\varepsilon$ solves the original regularised problem \eqref{eq:regularised pb}. As Lemma \ref{lem:to prove the prop} showed that $\bu^\varepsilon$ is smooth away from the initial time and solves \eqref{eq:modified regularised problem} in the classical sense, it follows that it is a classical solution of \eqref{eq:regularised pb}. 
\end{proof}

\subsection{Dissipation Estimate and $H^{-1}_\loc$ Compactness of Entropy Dissipation Measures}

In this subsection we derive the main dissipation estimate required for the pre-compactness in $H^{-1}_\loc$ of the entropy dissipation measures. Our main result is the following. 

\begin{lemma}[Dissipation Estimate]\label{lem:dissipation}
    Let $\bu^\varepsilon$ be a solution of the regularised problem \eqref{eq:regularised pb}. Then there exists a finite increasing function $C:(0,\infty) \to (0,\infty)$ independent of $\varepsilon$ such that, for all $t>0$, 
    \begin{equation}\label{eq:convex entropy estimate}
    \Vert \bu^\varepsilon\Vert^2_{L^\infty(0,t;L^2(\mathbb{R}))} + \Vert \sqrt{\varepsilon}\bu^\varepsilon_x \Vert^2_{L^2((0,t)\times\mathbb{R})} \leq C(t). 
\end{equation}
\end{lemma}

\begin{proof}
We divide the proof into two steps. 

\smallskip 
\noindent 1. \textit{Special entropy pair}. Select $\eta^*$ to be the special entropy of \eqref{eq:special entropy}, \textit{i.e.}
\begin{equation*}
    \eta^*(\sigma,\beta) = \frac{1}{2}(\sigma^2 + \beta^2), 
\end{equation*}
which has gradient given by $\nabla_\bu \eta^* = (\sigma,\beta)=\bu$ and Hessian $\nabla^2_\bu \eta^* = \mathrm{Id}$; it is strongly convex. It follows from testing \eqref{eq:regularised pb} with $(\nabla_\bu \eta^*)^\intercal$ that there holds 
\begin{equation}\label{eq:exp form ent eqn with viscosity}
    \eta^*(\bu^\varepsilon)_t + q^*(\bu^\varepsilon)_x = \varepsilon\nabla_\bu \eta^* \cdot \bu^\varepsilon_{xx} = \varepsilon\eta^*(\bu^\varepsilon)_{xx} - \underbrace{\varepsilon(\bu^\varepsilon_x)^\intercal \nabla^2_\bu \eta^*(\bu^\varepsilon) \bu_x^\varepsilon}_{=\varepsilon |\bu^\varepsilon_x|^2}. 
\end{equation}
Testing this equation against non-negative $\varphi \in C^\infty_c(\mathbb{R})$ depending only on $x$ and integrating by parts, we obtain 
\begin{equation*}
    \frac{\der}{\der t}\int_\mathbb{R}\varphi \eta^*(\bu^\varepsilon) \d x + \int_\mathbb{R} \varphi |\sqrt{\varepsilon}\bu^\varepsilon_x|^2 \d x = \int_\mathbb{R}  q^*(\bu^\varepsilon) \varphi_x \d x + \varepsilon \int_\mathbb{R} \eta^*(\bu^\varepsilon) \varphi_{xx} \d x, 
\end{equation*}
\textit{i.e.}, using the uniform bound on $\{\bu^\varepsilon\}_\varepsilon$ in $L^\infty(\mathbb{R}^2_+)$ and the continuity of $\eta^*$ and $q^*$, we get 
\begin{equation}\label{eq:convex entropy estimate not quite}
    \frac{1}{2}\int_\mathbb{R} \varphi |\bu^\varepsilon(t,x)|^2 \d x + \int_0^T \! \! \int_\mathbb{R} \varphi |\sqrt{\varepsilon}\bu^\varepsilon_x|^2 \d x \leq \frac{1}{2}\int_\mathbb{R} \varphi |\bu_0|^2 \d x + C\int_0^T \! \! \int_\mathbb{R} \! \! \big( |\varphi_x| + \varepsilon |\varphi_{xx}| \big) \d x \d t. 
\end{equation}

\smallskip 
\noindent 2. \textit{Relaxation of test function}. For all $R>0$, we define $\varphi^R$ as follows 
\begin{equation*}
    \varphi^R  \defeq  \mathds{1}_{[-R,R]}*\varrho_{\delta}, 
\end{equation*}
where $\varrho$ is the standard non-negative Friedrichs bump function supported on $[-1,1]$ and having unit mass, and $\varrho_{\delta}(x) = \delta^{-1} \varrho(x/\delta)$, with $\delta$ chosen sufficiently small such that $\varphi^R \equiv 1$ on $[-R,R]$ and $\varphi^R \in C^\infty_c((-R-1,R+1))$. It follows that there exists a positive constant $C$ independent of $R$ such that $$|\varphi^R_x|+|\varphi^R_{xx}| \leq C\mathds{1}_{[-R-1,-R]\cup[R,R+1]},$$
whence, inserting $\varphi^R$ as the test function in \eqref{eq:convex entropy estimate not quite}, the final term on the right-hand side of \eqref{eq:convex entropy estimate not quite} is estimated by 
\begin{equation*}
    \int_0^T \int_\mathbb{R}  \big( |\varphi^R_x| + \varepsilon |\varphi^R_{xx}| \big) \d x \d t \leq CT(1+\varepsilon) 
\end{equation*}
for some positive constant $C$ independent of $\varepsilon$ and $T$. By returning to \eqref{eq:convex entropy estimate not quite} and letting $R\to \infty$, the Monotone Convergence Theorem implies 
\begin{equation*}
    \frac{1}{2}\int_\mathbb{R} |\bu^\varepsilon(t,x)|^2 \d x + \int_0^T \int_\mathbb{R} |\sqrt{\varepsilon}\bu^\varepsilon_x|^2 \d x \d t \leq \frac{1}{2}\int_\mathbb{R} |\bu_0|^2 \d x + CT(1+\varepsilon). 
\end{equation*}
We deduce that there exists a positive constant $C$ independent of $\varepsilon$ and $T$ such that 
\begin{equation*}
    \Vert \bu^\varepsilon\Vert^2_{L^\infty(0,T;L^2(\mathbb{R}))} + \Vert \sqrt{\varepsilon}\bu^\varepsilon_x \Vert^2_{L^2((0,T)\times\mathbb{R})} \leq C(1+T). 
\end{equation*}
By replicating this strategy on each time-interval $[nT,(n+1)T]$ for $n \in \mathbb{N}$, we obtain \eqref{eq:convex entropy estimate}. 
\end{proof}

We are now in a position to prove the compactness in $H^{-1}_\loc$ of the entropy dissipation measures. This is encapsulated in the following lemma. 

\begin{lemma}[$H^{-1}_\loc$ Compactness of Entropy Dissipation Measures]\label{lem:H-1 cpct}
    Let $(\eta,q)$ be a $C^2$ entropy pair of the system \eqref{eq:consvn law carroll gamma 3} and $\{\bu^\varepsilon\}_\varepsilon$ be the sequence of regularised solutions provided by Proposition \ref{prop:approx pbms}. Then $\big\{\eta(\bu^\varepsilon)_t + q(\bu^\varepsilon)_x \big\}_\varepsilon$ is pre-compact in $H^{-1}_\loc(\mathbb{R}^2_+)$, \textit{i.e.}~for all compact subsets $K \subset \mathbb{R}^2_+$ there exists a compact set $\kappa_K \subset H^{-1}(K)$ such that 
    \begin{equation*}
        \big\{\eta(\bu^\varepsilon)_t + q(\bu^\varepsilon)_x \big\}_\varepsilon \subset \kappa_K. 
    \end{equation*}
\end{lemma}

\begin{proof}
    Recall from the computation \eqref{eq:exp form ent eqn with viscosity} that, for a general $C^2$ entropy pair, there holds 
    \begin{equation}\label{eq:general ent pair dissipation}
    \eta(\bu^\varepsilon)_t + q(\bu^\varepsilon)_x = \varepsilon\eta(\bu^\varepsilon)_{xx} - \varepsilon(\bu^\varepsilon_x)^\intercal \nabla^2_\bu \eta(\bu^\varepsilon) \bu_x^\varepsilon. 
\end{equation}
We divide the proof into three simple steps. 

\smallskip 
\noindent 1. The dissipation estimate of Lemma \ref{lem:dissipation} implies the estimate
\begin{equation}\label{eq:first term cpctness}
    \Vert \varepsilon \eta(\bu^\varepsilon)_{xx} \Vert_{H^{-1}(K)} \leq C_\eta \sqrt{\varepsilon} \Vert \sqrt{\varepsilon}\bu^\varepsilon_{x} \Vert_{L^2(K)} \leq C_{\eta,K} \sqrt{\varepsilon}. 
\end{equation}

\smallskip 
\noindent 2. Meanwhile, using the min-max theorem for the Rayleigh quotient of a symmetric matrix, we have
    \begin{equation}\label{eq:rayleigh quotient}
        |\varepsilon(\bu^\varepsilon_x)^\intercal \nabla^2_\bu \eta(\bu^\varepsilon) \bu_x^\varepsilon| \leq C_\eta |\sqrt{\varepsilon} \bu^\varepsilon_x|^2, 
    \end{equation}
   where we used that $\eta$ is $C^2$ and $\{\bu^\varepsilon\}_\varepsilon$ is uniformly bounded in $L^\infty$. The dissipation estimate of Lemma \ref{lem:dissipation} therefore implies 
    \begin{equation}\label{eq:second term cpctness}
        \Vert \varepsilon(\bu^\varepsilon_x)^\intercal \nabla^2_\bu \eta(\bu^\varepsilon) \bu_x^\varepsilon \Vert_{L^1(K)} \leq C_{\eta,K}, 
    \end{equation}
  and it follows that $\{\varepsilon(\bu^\varepsilon_x)^\intercal \nabla^2_\bu \eta(\bu^\varepsilon) \bu_x^\varepsilon \}_\varepsilon$ is pre-compact in $W^{-1,r}(K)$ for all $r \in (1,2)$ by the Rellich Theorem. 

 \smallskip 
 
 \noindent 3. Furthermore, using the uniform $L^\infty$ estimate on $\{\bu^\varepsilon\}_\varepsilon$, there holds 
    \begin{equation*}
      \Vert  \eta(\bu^\varepsilon)_t + q(\bu^\varepsilon)_x \Vert_{W^{-1,\infty}(K)} \leq C_{\eta,K}. 
    \end{equation*}
    Using the interpolation compactness lemma of Ding, Chen, and Luo \cite[Chapter 4]{dingchenluo1} (\textit{cf.} Murat's Lemma \cite{muratcone}), the estimates \eqref{eq:first term cpctness}--\eqref{eq:second term cpctness} now imply that the sequence $\{ \eta(\bu^\varepsilon)_t + q(\bu^\varepsilon)_x \}_\varepsilon$ is contained in $\kappa$ a compact subset of $H^{-1}(K)$. 
\end{proof}

In fact, we can give a more precise expression for the entropy dissipation $-\varepsilon(\bu^\varepsilon_x)^\intercal \nabla^2_\bu \eta(\bu^\varepsilon) \bu_x^\varepsilon$ when considering entropies generated by the kernels $(\chi_j,\varsigma_j)$ ($j=1,2$) as per Lemma \ref{lem:ent pairs}. We compute explicitly, with the notation $\eta_j$ ($j=1,2$) as in the statement of Lemma \ref{lem:ent pairs}: 
\begin{equation}\label{eq:first entropy dissipation explicit}
    -\varepsilon(\bu^\varepsilon_x)^\intercal \nabla^2_\bu \eta_1(\bu^\varepsilon) \bu_x^\varepsilon = -\varepsilon\Big( f''(\beta^\varepsilon-\sigma^\varepsilon) (\sigma^\varepsilon_x - \beta^\varepsilon_x)^2 + f''(\beta^\varepsilon + \sigma^\varepsilon) (\sigma^\varepsilon_x + \beta^\varepsilon_x)^2  \Big), 
\end{equation}
and 
\begin{equation}\label{eq:second entropy dissipation explicit}
    -\varepsilon(\bu^\varepsilon_x)^\intercal \nabla^2_\bu \eta_2(\bu^\varepsilon) \bu_x^\varepsilon = -\varepsilon\Big( -g'(\beta^\varepsilon-\sigma^\varepsilon)(\sigma^\varepsilon_x - \beta^\varepsilon_x)^2 + g'(\beta^\varepsilon+\sigma^\varepsilon)(\sigma^\varepsilon_x + \beta^\varepsilon_x)^2  \Big), 
\end{equation}
whence \eqref{eq:rayleigh quotient} follows immediately from the uniform $L^\infty$ bounds for $\{\bu^\varepsilon\}_\varepsilon$ and the regularity of $f$ and $g$, as claimed. The formulas \eqref{eq:first entropy dissipation explicit} and \eqref{eq:second entropy dissipation explicit} will be used in the proof of the kinetic formulation (Theorem \ref{thm:main_iii}).

\section{Proofs of the Main Theorems}\label{sec:proofs of global exis and weak conti}

This section is devoted to the proofs of the existence of a global entropy solution (Theorem \ref{thm:main_i}), the weak continuity of the system (Theorem \ref{thm:main_ii}), and the kinetic formulation (Theorem \ref{thm:main_iii}). 

\subsection{Global Existence}\label{sec:global exist}

\begin{proof}[Proof of Theorem \ref{thm:main_i}]

The uniform estimates provided by Proposition \ref{prop:approx pbms} as well as the $H^{-1}_\loc$ compactness result of Lemma \ref{lem:H-1 cpct} imply that the sequence $\{\bu^\varepsilon\}_\varepsilon$ of solutions to the regularised problems of Proposition \ref{prop:approx pbms} satisfy the assumptions of the compensated compactness framework established in Proposition \ref{lem:comp comp}. As such, there exists a subsequence, which we do not relabel, as well as $\bu = (\sigma,\beta) \in L^\infty(\mathbb{R}^2_+)$, such that 
\begin{equation}\label{eq:convergence for global existence proof}
    \bu^\varepsilon \to \bu \quad \text{a.e.~and strongly in $L^p_{\loc}(\mathbb{R}^2_+)$ for all } p \in [1,\infty); 
\end{equation}
additionally, the dissipation estimate of Lemma \ref{lem:dissipation} implies that $ \bu \in L^\infty_{\loc}(0,\infty;L^2(\mathbb{R}))$. Moreover, there holds, for any $\varphi \in C^\infty_c([0,\infty)\times\mathbb{R})$, 
\begin{equation*}
    \int_{\mathbb{R}^2_+} \big( \bu^\varepsilon \varphi_t + \bF(\bu^\varepsilon) \varphi_x \big) \d x \d t = -\int_\mathbb{R} \bu_0(x) \varphi(0,x) \d x + \varepsilon \int_{\mathbb{R}^2_+} \bu^\varepsilon_x \varphi_x \d x \d t. 
\end{equation*}
We are now ready to conclude that the limit $\bu$ is a solution of \eqref{eq:consvn law carroll gamma 3} in the following senses.

\smallskip 
\noindent 1. \textit{Distributional solution}. Let $\varphi \in C^\infty_c([0,\infty)\times\mathbb{R})$ be a test function. The continuity of $\bF$ in $\mathbb{R}^2 \setminus (\{\sigma=0\}\cup \{\beta = \pm \sigma\})$ and the invariant regions provided by Proposition \ref{prop:approx pbms}, as well as the convergence \eqref{eq:convergence for global existence proof} and the Dominated Convergence Theorem, imply that 
\begin{equation*}
   \lim_{\varepsilon \to 0} \int_{\mathbb{R}^2_+} \big( \bu^\varepsilon \varphi_t + \bF(\bu^\varepsilon) \varphi_x \big) \d x \d t = \int_{\mathbb{R}^2_+} \big( \bu \varphi_t + \bF(\bu) \varphi_x \big) \d x \d t, 
\end{equation*}
while the dissipation estimate of Lemma \ref{lem:dissipation} implies that
\begin{equation*}
    \bigg| \varepsilon \int_{\mathbb{R}^2_+} \bu^\varepsilon_x \varphi_x \d x \d t \bigg| \leq C_{\varphi} \sqrt{\varepsilon} \to 0. 
\end{equation*}
It therefore follows that $\bu=(\sigma,\beta)$ is a distributional solution of \eqref{eq:consvn law carroll gamma 3} for $\gamma=3$.

\smallskip 
\noindent 2. \textit{Entropy solution}. Let $(\eta,q)$ be any $C^2$ convex entropy pair. It follows from the dissipation relation \eqref{eq:general ent pair dissipation} and the convexity of $\eta$ that, given any non-negative test function $\varphi \in C^\infty_c(\mathbb{R}^2_+)$, there holds 
\begin{equation*}
    \int_{\mathbb{R}^2_+} \big( \eta(\bu) \varphi_t + q(\bu) \varphi_x \big) \d x \d t = \lim_{\varepsilon \to 0} \int_{\mathbb{R}^2_+} \big( \eta(\bu^\varepsilon) \varphi_t + q(\bu^\varepsilon) \varphi_x \big) \d x \d t \geq -\lim_{\varepsilon\to 0}\varepsilon \int_{\mathbb{R}^2_+} \eta(\bu^\varepsilon)_x \varphi_x \d x \d t, 
\end{equation*}
and the final term on the right-hand side is estimated using the dissipation estimate of Lemma \ref{lem:dissipation} as 
\begin{equation*}
    \bigg|\varepsilon \int_{\mathbb{R}^2_+} \eta(\bu^\varepsilon)_x \varphi_x \d x \d t \bigg| \leq C_{\eta,\varphi}\sqrt{\varepsilon} \to 0. 
\end{equation*}
It therefore follows that there holds, in the sense of distributions, 
\begin{equation*}
    \eta(\bu)_t + q(\bu)_x \leq 0. 
\end{equation*}
Both requirements of Definition \ref{def:ent sol} are fulfilled, which concludes the proof. 
\end{proof}

\subsection{Weak Continuity}\label{sec:weak continuity}

\begin{proof}[Proof of Theorem \ref{thm:main_ii}]
    Let $(\eta,q)$ be any $C^2$ entropy pair and define $\bu^n \defeq (\sigma^n,\beta^n)$ and 
    \begin{equation*}
        D^n \defeq \eta(\bu^n)_t + q(\bu^n)_x. 
    \end{equation*}
    We must show that $\{D^n\}_n$ is pre-compact in $H^{-1}_\loc(\mathbb{R}^2_+)$; once this is proved, we may directly apply Proposition \ref{lem:comp comp} to deduce that $\bu^n \to \bu$ a.e.~to some $\bu \in L^\infty(\mathbb{R}^2_+)$. The same strategy as that used in the Proof of Theorem \ref{thm:main_i} is then employed to show that $\bu$ is also an entropy solution of \eqref{eq:consvn law carroll gamma 3} for $\gamma=3$. We proceed in two steps.

    \smallskip 
    \noindent 1. Define $$ d^n \defeq \eta^*(\bu^n)_t + q^*(\bu^n)_x, $$
    where we recall that $(\eta^*,q^*)$ is the special entropy pair given by \eqref{eq:special entropy}. As this entropy pair is $C^2$, it follows directly from the definition of entropy solution that there holds 
    $$d^n \leq 0$$
    in the sense of distributions. Furthermore, since the sequence $\{\bu^n\}_n$ is uniformly bounded in $L^\infty$, we deduce that $\{d^n\}_n$ is pre-compact in $W^{-1,\infty}_\loc(\mathbb{R}^2_+)$. An application of Murat's Lemma \cite{muratcone} then implies that $$\{d^n\}_n \text{ is pre-compact in } W^{-1,r}_\loc(\mathbb{R}^2_+)$$ for all $r \in (1,\infty)$. The same argument implies that, given any $C^2$ \emph{convex} entropy pair $(\eta_c,q_c)$, the sequence $\{E^n\}_n$ determined by 
        \begin{equation*}
        E^n = \eta_c(\bu^n)_t + q_c(\bu^n)_x 
    \end{equation*}
    is pre-compact in $W^{-1,r}_\loc(\mathbb{R}^2_+)$ for all $r \in (1,\infty)$. 

    \smallskip 
    \noindent 2. Consider the sequence of entropy dissipation measures given by 
    \begin{equation*}
        F^n \defeq D^n + C_0 d^n 
    \end{equation*}
    for some positive constant $C_0$ to be determined. Note that, by the linearity of the entropy equation \eqref{eq:ent eqn}, the pair of functions determined by 
    \begin{equation*}
        \eta_c \defeq \eta + C_0 \eta^*, \quad q_c \defeq q + C_0 q^* 
    \end{equation*}
    is also an entropy pair of the system. Furthermore, since $\eta,q$ are $C^2$ and $\{\bu^n\}_n$ is uniformly bounded (\textit{i.e.}~there exists a bounded $B \subset \mathbb{R}^2$ such that $\bu^n(t,x) \in B$ for a.e.~$(t,x) \in \mathbb{R}^2_+$ for all $n$), there exists a positive constant $C_0$ such that 
    \begin{equation*}
        C_0 \, \mathrm{Id} + \nabla^2\eta \text{ is positive definite on } B. 
    \end{equation*}
As such, the entropy pair $(\eta_c,q_c)$ is convex on the range of all of the $\bu^n$, and hence the conclusion of Step 1 implies that 
\begin{equation*}
    \eta_c(\bu^n)_t + q_c(\bu^n)_x = F^n \text{ is pre-compact in } W^{-1,r}_\loc(\mathbb{R}^2_+) 
\end{equation*}
for all $r \in (1,\infty)$. Meanwhile, we already know from Step 1 that $\{d^n\}_n$ is pre-compact in $W^{-1,r}_\loc(\mathbb{R}^2_+)$, whence the linearity $D^n = F^n - C_0 d^n$ implies, in particular, that $$\{D^n\}_n \text{ is pre-compact in } H^{-1}_\loc(\mathbb{R}^2_+).$$
As per the proof of Theorem \ref{thm:main_i}, we deduce that $\bu^n \to \bu$ a.e.~and strongly in $L^p_\loc(\mathbb{R}^2_+)$ for all $p \in [1,\infty)$ and that $\bu$ is an entropy solution, which completes the proof.
\end{proof}

\subsection{Kinetic Formulation}\label{sec:proof of kinetic}

\begin{proof}[Proof of Theorem \ref{thm:main_iii}]
We follow the strategy of, \textit{e.g.}, \cite[\S 6]{Matthew}. By rewriting \eqref{eq:general ent pair dissipation} in terms of the kernels $(\chi_j,\varsigma_j)$ ($j=1,2$) and using the formulas \eqref{eq:first entropy dissipation explicit}--\eqref{eq:second entropy dissipation explicit}, we have for all $(f,g) \in \mathcal{A}\times\mathcal{B}$ 
    \begin{equation}\label{eq:dissipations in terms ofkernel 1}
       \begin{aligned} \partial_t\langle \chi_1(\sigma^\varepsilon(t,x),\beta^\varepsilon(t,x),\cdot) , f \rangle + \partial_x \langle \varsigma_1(\sigma^\varepsilon(t,x),\beta^\varepsilon(t,x),\cdot) , f \rangle = & \, \varepsilon \partial_{xx}\langle \chi_1(\sigma^\varepsilon(t,x),\beta^\varepsilon(t,x),\cdot),f \rangle \\ 
       &- \langle \partial^2_s\mu_1^\varepsilon(t,x,\cdot) , f \rangle 
       \end{aligned}
    \end{equation}
    and 
        \begin{equation}\label{eq:dissipations in terms ofkernel 2}
       \begin{aligned} \partial_t\langle \chi_2(\sigma^\varepsilon(t,x),\beta^\varepsilon(t,x),\cdot) , g \rangle + \partial_x \langle \varsigma_2(\sigma^\varepsilon(t,x),\beta^\varepsilon(t,x),\cdot) , g \rangle = \, &\varepsilon \partial_{xx}\langle \chi_2(\sigma^\varepsilon(t,x),\beta^\varepsilon(t,x),\cdot),g \rangle \\ 
       &- \langle \partial_s\mu_2^\varepsilon(t,x,\cdot) , g \rangle, 
       \end{aligned}
    \end{equation}
    where we have defined 
    \begin{equation}\label{eq:dissipations defis}
        \begin{aligned}
            &\mu_1^\varepsilon(t,x) \defeq \varepsilon \Big( \delta(\beta^\varepsilon-\sigma^\varepsilon-s) (\sigma^\varepsilon_x - \beta^\varepsilon_x)^2 + \delta(\beta^\varepsilon+\sigma^\varepsilon-s) (\sigma^\varepsilon_x + \beta^\varepsilon_x)^2 \Big), \\ 
            &\mu_2^\varepsilon(t,x) \defeq \varepsilon \Big(\delta(\beta^\varepsilon-\sigma^\varepsilon-s) (\sigma^\varepsilon_x - \beta^\varepsilon_x)^2 -\delta(\beta^\varepsilon+\sigma^\varepsilon-s) (\sigma^\varepsilon_x + \beta^\varepsilon_x)^2 \Big). 
        \end{aligned}
    \end{equation}
    We note in passing that a standard argument involving Riemann sums shows that we may interchange the derivatives with the duality brackets in equations \eqref{eq:dissipations in terms ofkernel 1}--\eqref{eq:dissipations in terms ofkernel 2}. By testing against any $\psi \in C^2_c(\mathbb{R}^2_+)$, we get 
    \begin{equation}\label{eq:before limit kinetic 1}
    \begin{aligned}
        -\int_{\mathbb{R}^2_+} \Big\langle \Big( \psi_t \chi_1(\sigma^\varepsilon(t,x),\beta^\varepsilon(t,x),\cdot) + \psi_x \varsigma_1&(\sigma^\varepsilon(t,x),\beta^\varepsilon(t,x),\cdot) , f \Big\rangle \d x \d t \\ 
        =& \, \varepsilon \int_{\mathbb{R}^2_+} \psi_{xx}\langle \chi_1(\sigma^\varepsilon(t,x),\beta^\varepsilon(t,x),\cdot) , f \rangle \d x \d t \\ 
        &- \int_{\mathbb{R}^2_+} \psi \langle \partial^2_s\mu_1^\varepsilon(t,x,\cdot), f \rangle \d x \d t, 
        \end{aligned}
    \end{equation}
    and 
        \begin{equation}\label{eq:before limit kinetic 2}
    \begin{aligned}
        -\int_{\mathbb{R}^2_+} \Big\langle \Big( \psi_t \chi_2(\sigma^\varepsilon(t,x),\beta^\varepsilon(t,x),\cdot) + \psi_x \varsigma_2&(\sigma^\varepsilon(t,x),\beta^\varepsilon(t,x),\cdot) , f \Big\rangle \d x \d t \\ 
        =& \, \varepsilon \int_{\mathbb{R}^2_+} \psi_{xx} \langle \chi_2(\sigma^\varepsilon(t,x),\beta^\varepsilon(t,x),\cdot) , g \rangle \d x \d t \\ 
        &- \int_{\mathbb{R}^2_+} \psi \langle \partial_s\mu_2^\varepsilon(t,x,\cdot), g \rangle \d x \d t, 
        \end{aligned}
    \end{equation}
    Observe that, for $j=1,2$, with $\phi \in C^0(\mathbb{R})$, using the explicit formulas for $\eta_1$ and $\eta_2$ in Lemma \ref{lem:ent pairs} as well as the uniform bounds $\sup_\varepsilon \Vert \bu^\varepsilon \Vert_{L^\infty(\mathbb{R}^2_+)} < \infty$, we have 
    \begin{equation*}
        \big| \langle \chi_j(\sigma^\varepsilon(t,x),\beta^\varepsilon(t,x),\cdot) , \phi \rangle \big| = |\eta_j^\psi(\bu^\varepsilon)| \leq C_\phi 
    \end{equation*}
    for some positive constant $C_\phi$ independent of $\varepsilon$. It therefore follows that, for any $\varphi_1 \in \mathcal{A}$ and $\varphi_2 \in \mathcal{B}$, there holds ($j=1,2$) 
    \begin{equation*}
        \bigg| \int_{\mathbb{R}^2_+} \psi_{xx} \langle \chi_j(\sigma^\varepsilon(t,x),\beta^\varepsilon(t,x),\cdot) , \varphi_j \rangle \d x \d t \bigg| \leq C_{\varphi_j} \varepsilon \Vert \psi_{xx} \Vert_{L^1(\mathbb{R}^2_+)} \to 0 \quad \text{as } \varepsilon \to 0. 
    \end{equation*}
    Meanwhile, the strong convergence $\bu^\varepsilon \to \bu$ pointwise almost everywhere implies that, for any $\varphi_1 \in \mathcal{A}$ and $\varphi_2 \in \mathcal{B}$, there holds ($j=1,2$) 
    \begin{equation*}
    \begin{aligned}
        \lim_{\varepsilon\to0}\int_{\mathbb{R}^2_+} \Big\langle \Big( \psi_t \chi_j&(\sigma^\varepsilon(t,x),\beta^\varepsilon(t,x),\cdot) + \psi_x \varsigma_j(\sigma^\varepsilon(t,x),\beta^\varepsilon(t,x),\cdot) , \varphi_j \Big\rangle \d x \d t \\ 
        &= \int_{\mathbb{R}^2_+} \Big\langle \Big( \psi_t \chi_j(\sigma(t,x),\beta(t,x),\cdot) + \psi_x \varsigma_j(\sigma(t,x),\beta(t,x),\cdot) , \varphi_j \Big\rangle \d x \d t. 
    \end{aligned}
    \end{equation*}
   Also, we see directly from the formulas \eqref{eq:dissipations defis} that the distributions $\mu^\varepsilon_j$ ($j=1,2$) are such that, given any $\phi \in C^0_c(\mathbb{R}^2_+\times\mathbb{R})$, 
\begin{equation}\label{eq:est for bounded radon}
    \bigg|\int_{\mathbb{R}^2_+} \langle \mu^\varepsilon_j(t,x,\cdot),\phi \rangle \d x \d t \bigg| \leq 2\Vert \phi \Vert_{L^\infty(\supp \phi)} \int_{\supp \phi} |\sqrt{\varepsilon}\bu^\varepsilon_x|^2 \d x \d t \leq C, 
\end{equation}
   for some positive constant $C$ independent of $\varepsilon$, where we used the uniform dissipation estimate of \Cref{lem:dissipation} to obtain the final inequality. It follows that the sequences $\{\mu^\varepsilon_j\}_\varepsilon$ ($j=1,2$) are uniformly bounded in the dual of $C^0_c(\mathbb{R}^2_+\times\mathbb{R})$. We deduce from Alaoglu's Theorem and the Markov--Kakutani Theorem that there exists $\mu_1$ and $\mu_2$ bounded Radon measures on $\mathbb{R}^2_+\times\mathbb{R}$ such that 
   \begin{equation*}
       \mu^\varepsilon_j \overset{*}{\rightharpoonup} \mu_j \text{ in } \big( C^0_c(\mathbb{R}^2_+\times\mathbb{R}) \big)'  \qquad (j=1,2); 
   \end{equation*}
  we note in passing that an analogous estimate to \eqref{eq:est for bounded radon} shows that this convergence also occurs in the sense of duality with $C^0_c(\mathbb{R}^2_+)\times\mathcal{A}$ for $\mu_1$, and in duality with $C^0_c(\mathbb{R}^2_+)\times\mathcal{B}$ for $\mu_2$. We may therefore pass to the limit in the right-hand sides of \eqref{eq:before limit kinetic 1}--\eqref{eq:before limit kinetic 2} and thereby obtain \eqref{eq:kinetic formulation 1}--\eqref{eq:kinetic formulation 2} as claimed. Finally, it is clear from the expression in \eqref{eq:dissipations defis} that $\mu_1$ is a non-negative measure, which completes the proof.
\end{proof}

\section*{Acknowledgements}

The present work was triggered during discussions at the \emph{Journ\'ees Relativistes de Tours}, organised by X. Bekaert, Y. Herfray, S. Solodukhin and M. Volkov, held at the Institut Denis Poisson in June 2023. The authors thank the Faculty of Mathematics at the University of Cambridge for hospitality and financial support in February 2024. We thank Nikolaos Athanasiou for useful discussions. Simon Schulz also acknowledges the support of Centro di Ricerca Matematica Ennio De Giorgi.

\printbibliography

\end{document}